\documentclass[11pt]{article}
\usepackage{tikz}
\usepackage{fullpage, url,amsmath,amsfonts,amssymb,mathtools,mathrsfs,graphicx,  algorithm, float, sansmath,epstopdf,color,caption,enumitem,tabularx}
\usepackage[final]{pdfpages}
\usepackage{amsthm}
\usetikzlibrary{automata,topaths}
\usetikzlibrary{decorations.pathreplacing,shapes.misc}
\usepackage{fancyhdr}

\usepackage{multirow}
\usetikzlibrary{calc,arrows}
\theoremstyle{plain}
\newtheorem{definition}{Definition}

\newtheorem{theorem}{Theorem}
\newtheorem{lemma}[theorem]{Lemma}

\newtheorem{corollary}[theorem]{Corollary}
\newtheorem{proposition}[theorem]{Proposition}
\newtheorem{claim}[theorem]{Claim}

\usepackage{blindtext}

\usepackage[colorlinks,linkcolor=blue,citecolor=red]{hyperref}

\newcommand{\R}{{\mathbb R}}
\newcommand{\E}{{\mathcal E}}

\newcommand{\X}{{\mathcal X}}

\begin{document}
\title{ A Lyapunov Analysis of Momentum Methods \\ in  Optimization}
\author{Ashia C. Wilson \qquad Benjamin Recht \qquad Michael I. Jordan\\\\University of California, Berkeley}
\date{\today}
\maketitle

\begin{abstract}
Momentum methods play a significant role in optimization. Examples include Nesterov's accelerated gradient method and the conditional gradient algorithm. Several momentum methods are provably optimal under standard oracle models, and all use a technique called {\em estimate sequences} to analyze their convergence properties.  The technique of estimate sequences has long been considered difficult to understand, leading many researchers to generate alternative, ``more intuitive'' methods and analyses. We show there is an equivalence between the technique of estimate sequences and a family of Lyapunov functions in both continuous and discrete time. This connection allows us to develop a simple and unified analysis of many existing momentum algorithms, introduce several new algorithms, and strengthen the connection between algorithms and continuous-time dynamical systems. 
\end{abstract}

\section{Introduction}
Momentum is a powerful heuristic for accelerating the convergence of optimization methods.    One can intuitively ``add momentum'' to a method by adding to the current step a weighted version of the previous step, encouraging the method to move along search directions that had been previously seen to be fruitful.  Such methods were first studied formally by Polyak~\cite{Polyak1964}, and have been employed in many practical optimization solvers.  As an example, since the 1980s, momentum methods have been popular in neural networks as a way to accelerate the backpropagation algorithm.  The conventional intuition is that momentum allows local search to avoid ``long ravines'' and ``sharp curvatures'' in the sublevel sets of cost functions~\cite{Rumelhardt}.

Polyak motivated momentum methods by an analogy to a ``heavy ball'' moving in a potential well defined by the cost function.  However, Polyak's physical intuition was difficult to make rigorous mathematically.  For quadratic costs, Polyak was able to provide an eigenvalue argument that showed that his Heavy Ball Method required no more iterations than the method of conjugate gradients~\cite{Polyak1964}.\footnote{Indeed, when applied to positive-definite quadratic cost functions, Polyak's Heavy Ball Method is equivalent to Chebyshev's Iterative Method~\cite{Chebyshev}.}   Despite its intuitive elegance, however, Polyak's eigenvalue analysis does not apply globally for general convex cost functions. In fact, Lessard~\emph{et al.} derived a simple one-dimensional counterexample where the standard Heavy Ball Method does not converge~\cite{Lessard14}.  

In order to make momentum methods rigorous, a different approach was required.  In celebrated work, Nesterov devised a general scheme to accelerate convex optimization methods, achieving optimal running times under oracle models in convex programming~\cite{Nesterov04}. To achieve such general applicability, Nesterov's proof techniques abandoned the physical intuition of Polyak~\cite{Nesterov04}; in lieu of differential equations and Lyapunov functions, Nesterov devised the method of \emph{estimate sequences} to verify the correctness of these momentum-based methods.  Researchers have struggled to understand the foundations and scope of the estimate sequence methodology since Nesterov's initial papers.  The associated proof techniques are often viewed as an ``algebraic trick.''  

To overcome the lack of fundamental understanding of the estimate sequence technique, 
several authors have recently proposed schemes to achieve acceleration without appealing 
to it~\cite{Fazel, BubeckLeeSingh15, Lessard14, DroriTeboulle13}.  One promising general 
approach to the analysis of acceleration has been to analyze the continuous-time 
limit of accelerated methods~\cite{SuBoydCandes,Krichene15}, or to derive these 
limiting ODEs directly via an underlying Lagrangian~\cite{Acceleration}, and
to prove that the ODEs are stable via a Lyapunov function argument.  However, these 
methods stop short of providing principles for deriving a discrete-time optimization 
algorithm from a continuous-time ODE.  There are many ways to discretize ODEs, but not 
all of them give rise to convergent methods or to acceleration.  Indeed, for unconstrained 
optimization on Euclidean spaces in the setting where the objective is strongly convex, 
Polyak's Heavy Ball method and Nesterov's accelerated gradient descent have the same 
continuous-time limit.  One recent line of attack on the discretization problem is 
via the use of a time-varying Hamiltonian and symplectic 
integrators~\cite{BetancourtJordanWilson}.  In this paper, we present 
a different approach, one based on a fuller development of Lyapunov theory.  
In particular, we present Lyapunov functions for both the continuous and 
discrete settings, and we show how to move between these Lyapunov functions.  
Our Lyapunov functions are time-varying and they thus allow us to establish 
rates of convergence.  They allow us to dispense with estimate sequences altogether, 
in favor of a dynamical-systems perspective that encompasses both continuous time
and discrete time.
\section{A Dynamical View of Momentum Methods}
\label{Sec:Dyn}
\paragraph{Problem setting.}
We are concerned with the following class of constrained optimization problems:
\begin{align}\label{eq:main-problem}
\min_{x \in \X} \; f(x),
\end{align}
where $\X \subseteq \R^d$ is a closed convex set and $f \colon \X \to \R$ is a continuously differentiable convex function. 
We use the standard Euclidean norm $\|x\| = \langle x,x \rangle^{1/2}$ throughout.
We consider the general non-Euclidean setting in which the space $\X$ is endowed with a distance-generating function $h \colon \X \to \R$ that is convex and essentially smooth (i.e., $h$ is continuously differentiable in $\X$, and $\|\nabla h(x)\|_* \to \infty$ as $\|x\| \to \infty$). The function $h$ can be used to define a measure of distance in $\X$ via its Bregman divergence:
\begin{align*}
D_h(y,x) = h(y) - h(x) - \langle \nabla h(x), y-x \rangle,
\end{align*}
which is nonnegative since $h$ is convex. 
The \emph{Euclidean setting} is obtained when $h(x) = \frac{1}{2} \|x\|^2$.

We denote a discrete-time sequence in lower case, e.g., $x_k$ with $k \ge 0$ an integer. We denote a continuous-time curve in upper case, e.g., $X_t$ with $t \in \R$. An over-dot means derivative with respect to time, i.e., $\dot X_t = \frac{d}{dt} X_t$.

\subsection{The Bregman Lagrangian}
Wibisono, Wilson and Jordan recently introduced the following function on curves,
\begin{align}\label{eq:breg}
\mathcal{L}(x, v, t) = e^{\alpha_t + \gamma_t}\left(D_h\left(x, x+ e^{-\alpha_t} v\right) - e^{\beta_t} f(x)\right),
\end{align}
where $x \in \X$, $v \in \mathbb{R}^d$, and $t \in \mathbb{R}$ represent position, velocity and time, respectively~\cite{Acceleration}. 
They called \eqref{eq:breg} the {\em Bregman Lagrangian}. The functions $\alpha,\beta, \gamma: \mathbb{R} \rightarrow \mathbb{R}$ are arbitrary smooth increasing functions of time that determine the overall damping of the Lagrangian functional, as well as the weighting on the velocity and potential function. 
They also introduced the following ``ideal scaling conditions,'' which are
needed to obtain optimal rates of convergence:
\begin{subequations}\label{Eq:IdeSca}
\begin{align}
\dot \gamma_t \,&=\, e^{\alpha_t}  \label{Eq:IdeScaGam}\\
\dot \beta_t \,&\leq\, e^{\alpha_t}   \label{Eq:IdeScaBet}.
\end{align}
\end{subequations}
Given $\mathcal{L}(x, v, t)$, we can define a functional on curves $\{X_t\, :\, t \in \mathbb{R}\}$ called the {\em action} via integration of the Lagrangian: $\mathcal{A}(X)  = \int_{\mathbb{R}} \mathcal{L}(X_t, \dot X_t, t)dt$. Calculation of the Euler-Lagrange equation, $\frac{\partial \mathcal{L}}{\partial x} (X_t, \dot X_t, t) =\frac{d}{dt}\frac{\partial \mathcal{L}}{\partial v} (X_t, \dot X_t, t)$, allows us to obtain a stationary point for the problem of finding the curve which minimizes the action. 
Wibisono, Wilson, and Jordan showed~\cite[(2.7)]{Acceleration} that under the first scaling condition~\eqref{Eq:IdeScaGam}, the Euler-Lagrange equation for the Bregman Lagrangian 
reduces to the following ODE:
\begin{align}\label{eq:el1}
\frac{d}{dt} \nabla h(X_t + e^{-\alpha_t} \dot X_t) = - e^{\alpha_t + \beta_t}\nabla f(X_t).
\end{align}

\paragraph{Second Bregman Lagrangian.}
We introduce a second function on curves,
 \begin{align}\label{eq:breg2}
\mathcal{L}(x, v, t) = e^{\alpha_t + \gamma_t + \beta_t}\left(\mu D_h\left(x, x+ e^{-\alpha_t} v\right) - f(x)\right),
\end{align}
using the same definitions and scaling conditions. The Lagrangian~\eqref{eq:breg2} places a different damping on the kinetic energy than in the original Bregman Lagrangian~\eqref{eq:breg}. 
\begin{proposition}\label{Prop:EL} 
Under the same scaling condition~\eqref{Eq:IdeScaGam}, the Euler-Lagrange equation for the second Bregman Lagrangian~\eqref{eq:breg2} reduces to:
 \begin{align}\label{eq:el2}
\frac{d}{dt} \nabla h(X_t + e^{-\alpha_t} \dot X_t) &=  \dot \beta_t \nabla h(X_t) - \dot \beta_t \nabla h(X_t + e^{-\alpha_t} \dot X_t) - \frac{e^{\alpha_t}}{\mu} \nabla f(X_t).
 \end{align}
 \end{proposition}
 We provide a proof of Proposition~\ref{Prop:EL} in Appendix~\ref{App:Prop_EL}. 
In what follows, we pay close attention to the special case of the dynamics 
in~\eqref{eq:el2} where $h$ is Euclidean and the damping $ \beta_t = \gamma t$ is linear:
 \begin{align}\label{eq:euc}
 \ddot X_t + 2\gamma\dot X_t + \frac{\gamma^2}{\mu}\nabla f(X_t) =  0.
 \end{align}
When $\gamma = \sqrt{\mu}$, we can discretize the dynamics in~\eqref{eq:euc} to obtain accelerated gradient descent in the setting where $f$ is $\mu$-strongly convex.

\subsection{Lyapunov function for the Euler-Lagrange equation}
\label{Sec:Lyap}
To establish a convergence rate associated with solutions to the Euler-Lagrange equation for both families of dynamics~\eqref{eq:el1} and~\eqref{eq:el2}, under the ideal scaling conditions, we use Lyapunov's method~\cite{Lyapunov}. Lyapunov's method is based on the idea of constructing a positive definite quantity $\E: \X \rightarrow \R$ which decreases along the trajectories of the dynamical system $\dot X_t = v(X_t)$:
\begin{equation*}
\frac{d}{dt}\E(X_t) =\langle \nabla \E(X_t), v(X_t)\rangle < 0.
\end{equation*}
The existence of such a \emph{Lyapunov function} guarantees that the dynamical system converges: if the function is positive yet strictly decreasing along all trajectories, then the dynamical system must eventually approach a region where $\E(X)$ is minimal.  If this region coincides with the stationary points of the dynamics, then all trajectories must converge to a stationary point.
We now discuss the derivation of {\em time-dependent} Lyapunov functions for dynamical systems with bounded level sets. The Lyapunov functions will imply {\em convergence rates} for dynamics~\eqref{eq:breg} and \eqref{eq:el2}.
\begin{proposition}\label{Prop:WeakLyap}
Assume $f$ is convex, $h$ is strictly convex, and  the second ideal scaling condition~\eqref{Eq:IdeScaBet} holds. The Euler-Lagrange equation \eqref{eq:el1} satisfies 
\begin{align}\label{eq:ineq}
\frac{d}{dt}\Big\{ D_h(x, X_t + e^{-\alpha_t} \dot X_t)\Big\} \leq - \frac{d}{dt}\Big\{e^{\beta_t}(f(X_t) - f(x))\Big\},
\end{align}
when  $x = x^\ast$. If the ideal scaling holds with equality, $\dot \beta_t = e^{\alpha_t}$, the solutions satisfy~\eqref{eq:ineq} for $\forall x \in \mathcal{X}$. 
Thus, 
\begin{align}\label{eq:Lyap1}
\mathcal{E}_t =  D_h(x, X_t + e^{-\alpha_t} \dot X_t) + e^{\beta_t} (f(X_t) - f(x))
\end{align}
is a Lyapunov function for dynamics~\eqref{eq:el1}.
\end{proposition}
A similar proposition holds for the second family of dynamics~\eqref{eq:breg2} under the additional assumption that $f$ is $\mu$-uniformly convex with respect to $h$:
\begin{align}\label{eq:unif} D_f(x,y) \geq \mu  D_h(x,y).\end{align}
When $h(x) = \frac{1}{2}\|x\|^2$ is the Euclidean distance, \eqref{eq:unif} is equivalent to the standard assumption that $f$ is $\mu$-strongly convex. 
Another special family is obtained when $h(x) = \frac{1}{p}\|x\|^p$, 
which, as pointed out by Nesterov~\cite[Lemma 4]{Nesterov08}, yields
a Bregman divergence that is $\sigma$-uniformly convex with respect 
to the $p$-th power of the norm:
\begin{align}\label{eq:unif2}
D_h(x,y) \geq \frac{\sigma}{p}\|x-y\|^p,
\end{align}
where $\sigma =2^{-p+2} $. Therefore, if $f$ is uniformly convex with respect to the Bregman divergence generated by the $p$-th power of the norm, it is also uniformly convex with respect to the $p$-th power of the norm itself. We are now ready to state the main proposition for the continuous-time dynamics.
\begin{proposition}\label{prop:strLyap}
Assume $f$ is $\mu$-uniformly convex with respect to $h$~\eqref{eq:unif}, $h$ is strictly convex,  and  the second ideal scaling condition~\eqref{Eq:IdeScaBet} holds. Using dynamics \eqref{eq:el2}, we have the following inequality:
\begin{align*}
\frac{d}{dt}\Big\{ e^{\beta_t} \mu D_h(x, X_t + e^{-\alpha_t} \dot X_t)\Big\} \leq -\frac{d}{dt}\Big\{e^{\beta_t}(f(X_t) - f(x))\Big\},
\end{align*}
for $x = x^\ast$. If the ideal scaling holds with equality, $\dot \beta_t = e^{\alpha_t}$, the inequality holds for $\forall x \in \mathcal{X}$. In sum, we can conclude that
\begin{align}\label{eq:Lyap2}
\mathcal{E}_t = e^{\beta_t} \left(\mu D_h(x, X_t + e^{-\alpha_t} \dot X_t) + f(X_t) - f(x) \right)
\end{align}
is a Lyapunov function for dynamics~\eqref{eq:el2}.
\end{proposition}
%
The proof of both results, which can be found in Appendix~\ref{App:deriv_lyap}, uses 
the fundamental theorem of calculus and basic properties of dynamics~\eqref{eq:el2}. 
Taking $x = x^\ast$ and writing the Lyapunov property $\mathcal{E}_t \leq \mathcal{E}_0$ explicitly,
\begin{equation}\label{eq:Guarantee1}
f(X_t) - f(x^\ast) \leq \frac{D_h(x^\ast, X_0 + e^{-\alpha_0} \dot X_0) + e^{\beta_0}(f(X_0) - f(x^\ast))}{e^{\beta_t}}
\end{equation}
for~\eqref{eq:Lyap1}, and 
\begin{equation}
f(X_t) - f(x^\ast) \leq \frac{e^{\beta_0}(\mu D_h(x^\ast ,X_0 +e^{-\alpha_0} \dot X_0) + f(X_0) - f(x^\ast))}{e^{\beta_t}},
\end{equation} 
for~\eqref{eq:Lyap2}, allows us to infer a $O(e^{-\beta_t})$ convergence rate for the function value for both families of dynamics~\eqref{eq:el1} and~\eqref{eq:el2}.

So far, we have introduced two families of dynamics~\eqref{eq:el1} and \eqref{eq:el2}
and illustrated how to derive Lyapunov functions for these dynamics which certify a 
convergence rate to the minimum of an objective function $f$ under suitable smoothness 
conditions on $f$ and $h$. Next, we will discuss how various discretizations of 
dynamics~\eqref{eq:el1} and \eqref{eq:el2} produce algorithms which are useful for 
convex optimization. A similar discretization of the Lyapunov functions~\eqref{eq:Lyap1} and \eqref{eq:Lyap2} will provide us with tools we can use to analyze these algorithms. 
We defer discussion of additional mathematical properties of the dynamics that we 
introduce---such as existence and uniqueness---to Appendix~\ref{Sec:Exist}.

\section{Discretization Analysis}
\label{Sec:Disc}
In this section, we illustrate how to map from continuous-time dynamics to discrete-time sequences. We assume throughout this section that the second ideal scaling~\eqref{Eq:IdeScaBet} holds with equality, $\dot \beta_t = e^{\alpha_t}$.
\paragraph{Explicit and implicit methods.}
Consider a general vector field $\dot X_t = v(X_t)$, where $v:\mathbb{R}^n \rightarrow \mathbb{R}^n$ is smooth. The explicit Euler method evaluates the vector field at the current point to determine a discrete-time step
\begin{align*}
\frac{x_{k+1} - x_k}{\delta} = \frac{X_{t+\delta} - X_t}{\delta} = v(X_t) = v(x_k).
\end{align*}
The implicit Euler method, on the other hand, evaluates the vector field at the future point
\begin{align*}
\frac{x_{k+1} - x_k}{\delta} = \frac{X_{t+\delta} - X_t}{\delta} = v(X_{t+\delta}) = v(x_{k+1}).
\end{align*}
An advantage of the explicit Euler method is that it is easier to implement in practice. The implicit Euler method has greater stability and convergence properties but requires solving an expensive implicit equation.
We evaluate what happens when we apply these discretization techniques to both families of dynamics~\eqref{eq:el1} and~\eqref{eq:el2}. To do so, we write these dynamics as systems 
of first-order equations. The implicit and explicit Euler method can be combined in four separate ways to obtain algorithms we can analyze; for both families, we provide results on several combinations of the explicit and implicit methods, focusing on the family that gives rise to accelerated methods. 

\subsection{Methods arising from the first Euler-Lagrange equation}
We apply the implicit and explicit Euler schemes to dynamics \eqref{eq:el1}, 
written as the following system of first-order equations:
 \begin{subequations}\label{Eq:EL}
 \begin{align}
 Z_t &= X_t + \frac{e^{\beta_t}}{\frac{d}{dt} e^{\beta_t}} \dot X_t \label{Eq:ELZ},\\
 \frac{d}{dt}\nabla h(Z_t) &= -\left(\frac{d}{dt} e^{\beta_t} \right)\nabla f(X_t)\label{Eq:ELH}.
 \end{align}
 \end{subequations}
Wibisono, Wilson and Jordan showed that the polynomial family $\beta_t = p\log t$ 
is the continuous-time limit of a family of accelerated disrete-time methods~\cite{Acceleration},  Here, we consider any parameter $\beta_t$ whose time derivative $\frac{d}{dt} e^{\beta_t} = (A_{k+1} - A_k)/\delta$ can be well-approximated by a discrete-time sequence $(A_i)_{i=1}^k$. The advantage of choosing an arbitrary time scaling $\delta$ is that it leads to a broad family of algorithms. To illustrate this, make the approximations $Z_t = z_k$, $X_t = x_k$, $\frac{d}{dt} \nabla h(Z_t) = \frac{\nabla h(z_{k+1}) - \nabla h(z_k)}{\delta}$, $\dot  X_t = \frac{d}{dt} X_t = \frac{x_{k+1} - x_k}{\delta}$, and denote $\tau_k = \frac{A_{k+1} - A_k}{A_k}:= \frac{\alpha_k}{A_k}$, so that $\frac{e^{\beta_t}}{\frac{d}{dt} e^{\beta_t}} = \delta/\tau_k$. With these approximations, we explore various combinations of the explicit and implicit discretizations. 
\paragraph{Implicit-Implicit-Euler.}
Written as an algorithm, the implicit Euler method applied to~\eqref{Eq:ELZ} and~\eqref{Eq:ELH} has the following update equations:
\begin{subequations}\label{Eq:AlgoForward12}
\begin{align}
z_{k+1} &= \underset{\substack{\,\,\,\,\,\,z\in\X\\ \,\,\,\,\,\,x =  \frac{\tau_{k}}{1+\tau_{k}} z + \frac{1}{1 +\tau_{k}} x_k}}{\text{arg\,\,\,min}} \left\{ A_k f(x) + \frac{1}{\tau_k}D_h\left(z, z_k\right)\right\},\label{eq:hard2}\\
x_{k+1} &= \frac{\tau_{k}}{1 + \tau_{k}}z_{k+1} +\frac{1}{1+\tau_{k}}  x_{k}.
\end{align}
\end{subequations}
We now state our main proposition for the discrete-time dynamics. 
\begin{proposition}\label{Prop:Forward1}
Using the discrete-time Lyapunov function, 
\begin{align}\label{eq:nlyap}
E_{k} =  D_h(x^\ast, z_k) + A_k(f(x_k) - f(x^\ast)),
\end{align}
the bound
$\frac{E_{k+1} - E_k}{\delta} \leq 0$ holds
for algorithm~\eqref{Eq:AlgoForward12}. 
\end{proposition}
In particular, this allows us to conclude a general $O(1/A_k)$ convergence rate for the implicit method~\eqref{Eq:AlgoForward12}. 
\begin{proof}
The implicit scheme~\eqref{Eq:AlgoForward12}, with the aforementioned discrete-time approximations, satisfies the following variational inequalities:
\begin{subequations}
\begin{align}
\nabla h(z_{k+1}) - \nabla h(z_k) &= -(A_{k+1}-A_k) \nabla f(x_{k+1})\label{eq:mirfor}\\
( A_{k+1} - A_k)z_{k+1} &= ( A_{k+1} - A_k) x_{k+1} +  A_{k}(x_{k+1} - x_k) .\label{eq:coupfor}
\end{align}
\end{subequations}
Using these identities, we have the following derivation:
\begin{align*}
E_{k+1} - E_k &= D_h(x, z_{k+1})  -  D_h(x, z_{k}) + A_{k+1}(f(x_{k+1}) - f(x)) - A_{k}(f(x_{k}) - f(x))\\
& = - \langle  \nabla h(z_{k+1}) - \nabla h(z_k), x - z_{k+1}\rangle - D_h(z_{k+1}, z_k) \\
&\quad+ A_{k+1}(f(x_{k+1}) - f(x))- A_{k}(f(x_{k}) - f(x))\\
& \overset{\eqref{eq:mirfor}}{=}  (A_{k+1}-A_k)\langle  \nabla f(x_{k+1}), x - z_{k+1}\rangle - D_h(z_{k+1}, z_k) \\
&\quad + A_{k+1}(f(x_{k+1}) - f(x))-A_{k}(f(x_{k}) - f(x))\\
&\overset{\eqref{eq:coupfor}}{=} (A_{k+1}-A_k)\langle  \nabla f(x_{k+1}), x - x_{k+1}\rangle + A_k\langle  \nabla f(x_{k+1}), x_{k} - x_{k+1}\rangle
\\&\quad - D_h(z_{k+1}, z_k)+ A_{k}(f(x_{k+1}) - f(x_k)) +  (A_{k+1} - A_k)(f(x_{k+1}) - f(x))\\
&\leq 0.
\end{align*}
 The inequality on the last line follows from the convexity of $f$ and the strict convexity of $h$. 
\end{proof}
\paragraph{Accelerated gradient family.}
We study families of algorithms which give rise to a family of accelerated methods. These methods can be thought of variations of the explicit Euler scheme applied to~\eqref{Eq:ELZ} and the implicit Euler scheme applied to~\eqref{Eq:ELH}.\footnote{Here we make the identification $\tau_k = A_{k+1} - A_k/A_{k+1}:= \alpha_k/A_{k+1}$.} The first family of methods can be written as the following general sequence:
\begin{subequations}\label{eq:agd}
\begin{align}
x_{k+1} &= \tau_k z_k + (1- \tau_k) y_k \label{eq:coup}\\
\nabla h(z_{k+1}) &= \nabla h(z_k) - \alpha_k \nabla f(x_{k+1})\label{eq:mir}\\
y_{k+1} &= \mathcal{G}(x)\label{eq:grad},
\end{align}
\end{subequations}
where $\mathcal{G}:\mathcal{X} \rightarrow \mathcal{X}$ is an arbitrary map whose domain is the previous state, $x = (x_{k+1}, z_{k+1}, y_k)$.  The second family can be written: 
\begin{subequations}\label{eq:agd2}
\begin{align}
x_{k+1} &= \tau_k z_k + (1- \tau_k) y_k \label{eq:coup1}\\
y_{k+1} &= \mathcal{G}(x)\label{eq:grad1}\\
\nabla h(z_{k+1}) &= \nabla h(z_k) - \alpha_k \nabla f(y_{k+1}),\label{eq:mir1}
\end{align}
\end{subequations}
where $\mathcal{G}:\mathcal{X} \rightarrow \mathcal{X}$ is an arbitrary map whose domain is the previous state, $x = (x_{k+1}, z_{k}, y_k)$. When $\mathcal{G}(x) = x_{k+1}$ for either algorithm, we recover a classical explicit discretization applied to \eqref{Eq:ELZ} and implicit discretization applied to \eqref{Eq:ELH}.  We will show that the additional sequence $y_{k}$ allows us to obtain better error bounds in our Lyapunov analysis. Indeed,
we will show that accelerated gradient descent~\cite{Nesterov04,Nesterov05}, 
accelerated higher-order methods~\cite{Nesterov08,Baes09}, accelerated universal 
methods~\cite{Nesterov17}, accelerated proximal 
methods~\cite{Tseng08,BeckTeboulle09,Nesterov13} all involve particular
choices for the map $\mathcal{G}$ and for the smoothness assumptions on 
$f$ and $h$. Furthermore, we demonstrate how the analyses contained in 
all of these papers implicitly show the following discrete-time Lyapunov 
function,
\begin{align}\label{eq:lyap}
E_k = D_h(x^\ast, z_k) + A_k (f(y_k) - f(x^\ast)), 
\end{align}
is decreasing for each iteration $k$. To show this, we begin with the following 
proposition.
\begin{proposition} \label{Prop:WeakBound} Assume that the distance-generating 
function $h$ is $\sigma$-uniformly convex with respect to the $p$-th power of 
the norm $(p\geq 2)$ \eqref{eq:unif2} and the objective function $f$ is convex. 
Using only the updates \eqref{eq:coup} and \eqref{eq:mir}, and using the Lyapunov 
function~\eqref{eq:lyap}, we have the following bound:  
\begin{align}\label{eq:bound}
\frac{E_{k+1} - E_k}{\delta} \leq \varepsilon_{k+1},
\end{align}
where the error term scales as
\begin{subequations}\label{eq:error}
\begin{align}\label{eq:err}
\varepsilon_{k+1} =  \frac{p-1}{p}\sigma^{-\frac{1}{p-1}}\frac{(A_{k+1} - A_k)^{\frac{p}{p-1}}}{\delta}\|\nabla f(x_{k+1})\|^{\frac{p}{p-1}}   + \frac{A_{k+1}}{\delta}(f(y_{k+1}) -f(x_{k+1})). 
\end{align}
If we use the updates  \eqref{eq:coup1} and \eqref{eq:mir1} instead, the error term 
scales as
\begin{align}\label{eq:err1}
\varepsilon_{k+1} =  \frac{p-1}{p}\sigma^{-\frac{1}{p-1}}\frac{(A_{k+1} - A_k)^{\frac{p}{p-1}}}{\delta}\|\nabla f(y_{k+1})\|^{\frac{p}{p-1}}   + \frac{A_{k+1}}{\delta}\langle \nabla f(y_{k+1}), y_{k+1} - x_{k+1}\rangle. 
\end{align}
\end{subequations} 
\end{proposition}
The error bounds in~\eqref{eq:error} were obtained using no smoothness assumption 
on $f$ and $h$; they also hold when full gradients of $f$ are replaced with elements 
in the subgradient of $f$. The proof of this proposition can be found in 
Appendix~\ref{App:WeakBound}. 
The bounds in Proposition~\ref{Prop:WeakBound} were obtained {\em without} using the arbitrary update $y_{k+1} = \mathcal{G}(x)$. In particular, accelerated methods are obtained by picking a map $\mathcal{G}$ that results in a better bound on the error than the straightforward discretization $y_{k+1} = x_{k+1}$. We immediately see that any algorithm for which the map  $\mathcal{G}$ satisfies the progress condition $f(y_{k+1}) - f(x_{k+1}) \propto - \|\nabla f(x_{k+1})\|^{\frac{p}{p-1}}$ or $\langle \nabla f(y_{k+1}), y_{k+1} - x_{k+1}\rangle \propto - \|\nabla f(y_{k+1})\|^{\frac{p}{p-1}}$ will have a $O(1/\epsilon \sigma k^p)$ convergence rate. We now show how this general analysis applied concretely to each of
the aforementioned five methods. 
\paragraph{\bf Quasi-monotone method~\cite{Nesterov15}.}  The quasi-monotone subgradient method, which uses the map 
\begin{align*}
\mathcal{G}(x) = x_{k+1}
\end{align*} 
for both algorithms \eqref{eq:agd} and \eqref{eq:agd2}, was introduced by Nesterov in 2015.
Under this map, assuming the strong convexity of $h$ (which implies $p=2$), 
we can write the error~\eqref{eq:error} as
\begin{align}\label{eq:quas_bound}
\varepsilon_{k+1} = \frac{(A_{k+1} - A_k)^2}{2\sigma\delta}\|\nabla f(x_{k+1})\|^2. 
\end{align}
If we assume all the (sub)gradients of $f$ are upper bounded in norm, then maximizing $\sum_{i=1}^k \varepsilon_i/A_i$ results in an $O(1/\sqrt{k})$ convergence rate. This matches the lower bound for (sub)gradient methods designed for Lipschitz-convex functions.\footnote{The same convergence bound can be shown to hold for the (sub)gradient method under this smoothness class, when one assesses convergence for the average/minimum iterate~\cite{Nesterov04}.} 
\paragraph{\bf Accelerated gradient/mirror descent~\cite{Nesterov04,Nesterov05}.} In 1983, Nesterov introduced accelerated gradient decent, which uses the following family of operators $\mathcal{G} \equiv \mathcal{G}_\epsilon$, parameterized by a scaling constant $\epsilon>0$:
\begin{align}\label{eq:gradmap}
\mathcal{G}_\epsilon (x) = \arg \min_{y \in \mathcal{X}} \left\{ f(x) + \langle \nabla f(x), y-x\rangle + \frac{1}{2\epsilon}\|y-x\|^2 \right\}.
\end{align} 
Nesterov assumed the use of full gradients $\nabla f$ which are $(1/\epsilon)$-smooth; thus, the gradient map is scaled according to the Lipschitz parameter.
\begin{lemma} \label{eq:prop1} 
Assume $h$ is $\sigma$-strongly convex and $f$ is $(1/\epsilon)$-smooth. 
Using the gradient update, $y_{k+1} = \mathcal{G}_\epsilon (x_{k+1})$, for 
updates ~\eqref{eq:grad} and~\eqref{eq:grad1}, where $\mathcal{G}_\epsilon$ 
is defined in ~\eqref{eq:gradmap}, 
the error for algorithm \eqref{eq:agd} 
can be written as follows:
\begin{subequations}\label{eq:errr}
\begin{align}\label{eq:errra}
\varepsilon_{k+1} = \frac{(A_{k+1} - A_k)^2}{2\sigma\delta}\|\nabla f(x_{k+1})\|^2  -\frac{ \epsilon A_{k+1}}{2\delta}\|\nabla f(x_{k+1})\|^2,
\end{align}
and for algorithm \eqref{eq:agd2}, we have: 
\begin{align}\label{eq:errrb}
\varepsilon_{k+1} =  \frac{(A_{k+1} - A_k)^2}{2\sigma\delta}\|\nabla f(y_{k+1})\|^2 -\frac{ \epsilon A_{k+1}}{2\delta}\|\nabla f(y_{k+1})\|^2. 
\end{align}
\end{subequations}
\end{lemma}
\begin{proof}
The optimality condition for the gradient update~\eqref{eq:gradmap} is 
 \begin{align}\label{eq:one}
 \nabla f(x) = \frac{1}{\epsilon}(x -\mathcal{G}_\epsilon(x) ).
 \end{align}
The bound~\eqref{eq:errra} follows 
 from smoothness of the objective function $f$,
\begin{align*}
f(\mathcal{G}_{\epsilon}(x)) &\leq  f(x) + \langle \nabla f(x), \mathcal{G}_{\epsilon}(x)- x\rangle +  \frac{1}{2\epsilon}\| \mathcal{G}_{\epsilon}(x) - x\|^2\\
&\overset{\eqref{eq:one}}{=}f(x) - \frac{\epsilon}{2}\|\nabla f(x)\|^2.
\end{align*}
For the second bound~\eqref{eq:errrb}, 
we use the $(1/\epsilon)$-smoothness of the gradient, 
 \begin{align}\label{eq:two}
 \|\nabla f(\mathcal{G}_\epsilon(x)) - \nabla f(x)\|\leq \frac{1}{\epsilon}\|\mathcal{G}_\epsilon(x)-x\|;
 \end{align}
substituting \eqref{eq:one} into \eqref{eq:two}, squaring both sides, and expanding the square on the left-hand side, yields the desired bound:
 \begin{align*}
 \langle \nabla f(\mathcal{G}_\epsilon(x)), x - \mathcal{G}_\epsilon(x))\rangle \leq -\frac{\epsilon}{2}\|\nabla f(\mathcal{G}_\epsilon(x))\|^2.
\end{align*}
\end{proof}
The error bounds we have just obtained depend explicitly on the scaling $\epsilon$. 
This restricts our choice of sequences $A_k$; they must satisfy the following inequality:
\begin{align}\label{eq:gradbound}
\frac{(A_{k+1} - A_k)^2}{A_{k+1}} \leq\epsilon\sigma,
\end{align} 
for the error to be bounded. 
Choosing $A_k$ to be a polynomial in $k$ of degree two, with leading coefficients 
$\epsilon\sigma$, optimizes the bound \eqref{eq:gradbound}; from this we can conclude $f(y_k) - f(x^\ast) \leq O(1/\epsilon\sigma k^2)$, which matches the lower bound for algorithms which only use full gradients of the objective function. Furthermore, if we take the discretization step to scale according to the smoothness as $\delta = \sqrt{\epsilon}$, then both $\|x_k - y_k\| = O(\sqrt{\epsilon})$ and  $\varepsilon_k = O(\sqrt{\epsilon})$; therefore, as $\sqrt{\epsilon} \rightarrow 0$, we recover the  dynamics \eqref{Eq:EL} and the statement $\dot{\mathcal{E}}_t \leq 0$ for Lyapunov function \eqref{eq:el1} in the limit.

\paragraph{\bf Accelerated universal methods~\cite{Nesterov08,Baes09,Nesterov14,Nesterov17}.}\label{Accgrad}
The term ``universal methods'' refers to the algorithms designed for the class of functions with $(\epsilon,\nu)$-H\"older-continuous higher-order gradients ($2 \leq p \in \mathbb{N}$, $\nu \in (0,1]$, $\epsilon >0$),
\begin{equation}\label{eq:Hold}
\|\nabla^{p-1} f(x) - \nabla^{p-1} f(y)\| \leq \frac{1}{\epsilon}\|x - y\|^\nu.
\end{equation}
Typically, practitioners care about the setting where we have H\"older-continuous gradients ($p=2$) or H\"older-continuous Hessians ($p=3$), since methods which use higher-order information are often too computationally expensive. In the case $p\geq 3$, the gradient update
\begin{align}\label{eq:taylor3}
\mathcal{G}_{\epsilon, p, \nu,N}(x) = \arg \min_{y\in \X} \left\{f_{p-1}(x; y) + \frac{N}{\epsilon \tilde p}\|x - y\|^{\tilde p}\right\},\quad \tilde p = p-1+\nu, \, N >1
\end{align}
can be used to simplify the error~\eqref{eq:err1} obtained by algorithm~\eqref{eq:agd2}. Notice, the gradient update is regularized by the smoothness parameter $\tilde p$. We summarize this result in the following proposition. 
\begin{lemma}\label{prop:Hold}
Assume $f$ has H\"older-continuous higher-order gradients. Using the map
 $y_{k+1} = \mathcal{G}_{\epsilon, p, \nu,N}(x_{k+1})$, defined by \eqref{eq:taylor3}, 
in update~\eqref{eq:grad1} yields the following progress condition:
\begin{align}\label{eq:progcond}
\langle \nabla f(y_{k+1}), y_{k+1} - x_{k+1}\rangle  \leq -\frac{ (N^2 -1)^{\frac{\tilde p-1}{2\tilde p-2}}}{2N}\epsilon^{\frac{1}{\tilde p-1}}\ \|\nabla f(y_{k+1})\|^{\frac{\tilde p}{\tilde p-1}},
\end{align}
where  $\tilde p = p-1+\nu$ and $p \geq 3$.
\end{lemma}
Lemma~\ref{prop:Hold} demonstrates that if the Taylor approximation is regularized according to the smoothness of the function, the progress condition scales as a function of the smoothness in a particularly nice way. Using this inequality, we can simplify the error~\eqref{eq:err1} in algorithm~\eqref{eq:agd2} to the following,
\begin{align*}
\varepsilon_{k+1} &= \frac{\tilde p - 1}{\tilde p} \sigma^{-\frac{1}{\tilde p-1}}\frac{(A_{k+1} - A_k)^{\frac{\tilde p}{\tilde p-1}}}{\delta}\|\nabla f(y_{k+1})\|^{\frac{\tilde p}{\tilde p-1}}\\
&\quad -  \frac{A_{k+1}}{\delta}\frac{ (N^2 -1)^{\frac{\tilde p-1}{2\tilde p-2}}}{2N}\epsilon^{\frac{1}{\tilde p-1}}\|\nabla f(y_{k+1})\|^{\frac{\tilde p}{\tilde p-1}}, 
\end{align*}
where we have assumed that the geometry scales nicely with the smoothness condition: $D_h(x,y) \geq \frac{\sigma}{\tilde p}\|x - y\|^{\tilde p}$. This requires the condition $p \geq 3$.  To ensure a non-positive error we choose a sequence which satisfies the bound, 
\begin{align*}
\frac{(A_{k+1} - A_k)^{\frac{\tilde p}{\tilde p-1}}}{A_{k+1}} \leq (\epsilon\sigma)^{\frac{1}{\tilde p-1}}\frac{ \tilde p}{\tilde p-1}\frac{(N^2 -1)^{\frac{\tilde p-1}{2\tilde p-2}}}{2 N}:= C_{\epsilon, \sigma,\tilde p,N}.
\end{align*}
This bound is maximized by polynomials in $k$ of degree $\tilde p $ with leading coefficient proportional to $C_{\epsilon, \sigma,\tilde p,N}^{\tilde p -1}$; this results in the convergence rate bound $f(y_k) - f(x^\ast) \leq O(1/\epsilon \sigma k^{\tilde p}) = O(1/\epsilon\sigma k^{p-1+\nu})$. We can compare this convergence rate to that obtained by using just the gradient map $y_{k+1}= \mathcal{G}_{\epsilon, p, \tilde p,N}(y_{k})$; this algorithm yields a slower $f(y_k) - f(x^\ast) \leq O(1/\epsilon\sigma k^{\tilde p -1})= O(1/\epsilon \sigma k^{p-2+\nu})$ convergence rate under the same smoothness assumptions. Proofs of these statements can be found in Appendix~\ref{App:Hold}. 
This result unifies and extends the analyses of the accelerated (universal) cubic regularized Newton's method~\cite{Nesterov08, Nesterov17} and accelerated higher-order methods~\cite{Baes09}. Wibisono et al.~\cite{Acceleration} show that $\|x_k - y_k\| = O(\epsilon^{1/\tilde p})$ and $\varepsilon_k = O(\epsilon^{1/\tilde p})$ so that as $\epsilon^{1/\tilde p}\rightarrow 0$ we recover the  dynamics \eqref{Eq:EL} and the statement $\dot{\mathcal{E}}_t \leq 0$ for Lyapunov function \eqref{eq:el1}.

We end by mentioning that in the special case $p=2$, Nesterov~\cite{Nesterov14} showed that a slightly modified gradient map,
\begin{align}\label{eq:grad4}
\mathcal{G}_{\tilde \epsilon}(x) = x - \tilde \epsilon\, \nabla f(x),
\end{align}
has the following property when applied to functions with H\"older-continuous gradients.
\begin{lemma} \label{lem:Nest}(\cite[Lemma 1]{Nesterov14}) 
Assume $f$ has $(\epsilon,\nu)$-H\"older-continuous gradients, where $\nu \in (0,1]$. Then for $1/\tilde \epsilon \geq (1/2\tilde \delta)^{\frac{1-\nu}{1+\nu}}(1/\epsilon)^{\frac{2}{1+\nu}}$ the following bound: 
\begin{align*}
f(y_{k+1}) - f(x_{k+1}) \leq - \frac{\tilde \epsilon}{2}\|\nabla f(x_{k+1})\|^2 + \tilde \delta, 
\end{align*}
holds for $y_{k+1} = \mathcal{G}_{\tilde{\epsilon}}(x_{k+1})$ given by \eqref{eq:grad4}.
\end{lemma}
That is, if we take a gradient descent step with increased regularization and assume $h$ is $\sigma$-strongly convex, the error for algorithm~\eqref{eq:agd} when $f$ is $(\epsilon,\nu)$-H\"older-continuous can be written as,
\begin{align}
\varepsilon_{k+1} =  \frac{(A_{k+1} - A_k)^2}{2\sigma\delta}\|\nabla f(x_{k+1})\|^2 -\frac{\tilde  \epsilon A_{k+1}}{2\delta}\|\nabla f(x_{k+1})\|^2 + \tilde \delta.
\end{align}
This allows us to infer a $O(1 /\tilde{\epsilon}\sigma k^2)$ convergence rate of the function to within $\tilde \delta$, which is controlled by the amount of regularization $\tilde \epsilon$ we apply in the gradient update. 
%
%
Having discussed algorithms ``derived'' from dynamics \eqref{eq:el1}, we next discuss algorithms arising from the second family of dynamics~\eqref{eq:el2} and a proximal variant of it. The derivations and analyses will be remarkably similar to those presented in this section.
\subsection{Methods arising from the second Euler-Lagrange equation}
We apply the implicit and explicit Euler schemes to the dynamics \eqref{eq:el2} written as the following system of equations:
 \begin{subequations}\label{eq:EL1}
 \begin{align}
 Z_t &= X_t + \frac{e^{\beta_t}}{\frac{d}{dt} e^{\beta_t}} \dot X_t \label{Eq:ELZ1},\\
 \frac{d}{dt} \nabla h(Z_t) &= \frac{\frac{d}{dt} e^{\beta_t}}{e^{\beta_t}}\left( \nabla h(X_t) - \nabla h(Z_t) - \frac{1}{\mu} \nabla f(X_t)\right)\label{Eq:ELH1},
 \end{align}
 \end{subequations}
As in the previous setting, we consider any parameter $\beta_t$ whose time derivative $\frac{d}{dt} e^{\beta_t} = (A_{k+1} - A_k)/\delta$ can be well-approximated by a discrete-time sequence $(A_i)_{i=1}^k$. In addition, we make the discrete-time approximations  $\frac{d}{dt} \nabla h(Z_t) = \frac{\nabla h(z_{k+1}) - \nabla h(z_k)}{\delta}$ and $\frac{d}{dt} \dot X_t = \frac{x_{k+1} - x_k}{\delta}$, and denote $\tau_k = \frac{A_{k+1} - A_k}{A_k}$. 
We have the following proposition. 
 \begin{proposition}\label{Prop:Forward2}
Written as an algorithm, the implicit Euler scheme applied to~\eqref{Eq:ELZ1} and~\eqref{Eq:ELH1} results in the following updates:
\begin{subequations}\label{Eq:AlgoForward121}
\begin{align}
z_{k+1} &= \underset{\substack{\,\,\,\,\,\,z\in\X\\ \,\,\,\,\,\,x =  \frac{\tau_{k}}{1+\tau_{k}} z + \frac{1}{1 +\tau_{k}} x_k}}{\text{arg\,\,\,min}} \left\{ f(x) + \mu D_h(z,x) + \frac{\mu}{\tau_{k}} D_h\left(z, z_k\right)\right\},\label{eq:hard21}\\
x_{k+1} &= \frac{\tau_{k}}{1 + \tau_{k}}z_{k+1} +\frac{1}{1+\tau_{k}}  x_{k}.
\end{align}
\end{subequations}
Using the following discrete-time Lyapunov function:
\begin{align}\label{eq:lyapddd}
E_{k} =  A_k(\mu D_h(x^\ast, z_k) + f(x_k) - f(x^\ast)),
\end{align}
we obtain the bound
$E_{k+1} - E_k \leq 0$
for algorithm~\eqref{Eq:AlgoForward12}. This allows us to conclude a general $O(1/A_k)$ convergence rate for the implicit scheme~\eqref{Eq:AlgoForward12}. 
\end{proposition}
\begin{proof}
The algorithm that follows from the implicit discretization of the dynamics~\eqref{Eq:AlgoForward121} satisfies the variational conditions
\begin{subequations}
\begin{align}
\nabla h(z_{k+1}) - \nabla h(z_k) &= \tau_k\Big(\nabla h(x_{k+1}) - \nabla h(z_{k+1}) -  \frac{1}{\mu}\nabla f(x_{k+1})\Big)\label{eq:mirfor1}\\
(x_{k+1} - x_k) &= \tau_k(z_{k+1} - x_{k+1}),
\label{eq:coupfor2}
\end{align}
\end{subequations}
where $\tau_k = \frac{\alpha_k}{A_k}$. Using these variational inequalities, we have the following argument: 
\begin{align*}
E_{k+1} - E_k &= \alpha_k \mu D_h(x, z_{k+1}) + A_{k}\mu D_h(x, z_{k+1}) - A_k\mu D_h(x,z_k)  \\
&\quad + A_{k+1} (f(x_{k+1}) - f(x))- A_{k} (f(x_{k}) - f(x))\\
 &= \alpha_k \mu D_h(x, z_{k+1}) - A_{k} \mu \langle \nabla h(z_{k+1}) - \nabla h(z_k), x - z_{k+1}\rangle  -\mu A_kD_h(z_{k+1}, z_k) \\
&\quad+ A_{k+1} (f(x_{k+1}) - f(x)) - A_{k} (f(x_{k}) - f(x))\\
 &\overset{\eqref{eq:mirfor1}}{=} \alpha_k\mu D_h(x, z_{k+1}) +A_{k} \tau_k\langle \nabla f(x_{k+1}), x - z_{k+1}\rangle - A_{k} \mu D_h(z_{k+1}, z_k))\\
&\quad + A_{k} \langle \nabla f(x_{k+1}), x_k - x_{k+1}\rangle +  A_{k} \tau_k\mu \langle \nabla h(x_{k+1}) - \nabla h(z_{k+1}), x - z_{k+1}\rangle\\
&\quad+ A_{k+1} (f(x_{k+1}) - f(x))  - A_{k} (f(x_{k}) - f(x))  \\
&\overset{\eqref{eq:coupfor2}}{=} \alpha_k \mu D_h(x, z_{k+1}) +\alpha_k\langle \nabla f(x_{k+1}), x - x_{k+1}\rangle - A_{k} \mu D_h(z_{k+1}, z_k))  \\
&\quad + A_{k} \langle \nabla f(x_{k+1}), x_k - x_{k+1}\rangle +  \alpha_k \mu \langle \nabla h(x_{k+1}) - \nabla h(z_{k+1}), x - z_{k+1}\rangle \\
&\quad + A_{k} (f(x_{k+1}) - f(x_k))+ \alpha_{k} (f(x_{k+1}) - f(x)) \\
&\leq -\alpha_k \mu D_h(x_{k+1}, z_{k+1}) - A_{k}\mu D_h(z_{k+1}, z_k)
\end{align*}
The inequality uses the Bregman three-point identity~\eqref{eq:bregthree} and $\mu$-uniform convexity of $f$ with respect to $h$ \eqref{eq:unif}. 
\end{proof}

We now focus on analyzing the accelerated gradient family, which can be viewed as a discretization that contains easier subproblems. 
 
\subsubsection{Accelerated gradient descent~\cite{Nesterov04}}
We study a family of algorithms which can be thought of as slight variations of the implicit Euler scheme applied to~\eqref{Eq:ELZ1} and the explicit Euler scheme applied to~\eqref{Eq:ELH1}
\begin{subequations}\label{Eq:StrongConv21}
\begin{align}
x_k &= \frac{\tau_k}{1+\tau_k} z_k + \frac{1}{1+\tau_k} y_k \label{Eq:Coupling1}\\
\nabla h(z_{k+1}) - \nabla h(z_k) &= \tau_k \left( \nabla h(x_k) - \nabla h(z_k) - \frac{1}{\mu} \nabla f(x_k)\right) \label{Eq:ZSeq1}\\
y_{k+1} &= \mathcal{G}(x), \label{Eq:Grad1}
\end{align}
\end{subequations}
where $x = (x_{k}, z_{k+1}, y_k)$ is the previous state and  $\tau_k = \frac{A_{k+1} - A_k}{A_{k+1}}$. Note that when $\mathcal{G}(x) = x_{k}$, we recover classical discretizations. The additional sequence $y_{k+1} = \mathcal{G}(x)$, however, allows us to obtain better error bounds using the Lyapunov analysis. To analyze the general algorithm~\eqref{Eq:StrongConv21}, we use the following Lyapunov function:
\begin{align}\label{eq:lyapdd}
E_{k} = A_k(\mu D_h(x^\ast, z_k)+ f(y_k) - f(x^\ast)). 
\end{align}
We begin with the following proposition, which provides an initial error bound for algorithm~\eqref{Eq:StrongConv21} using the general update \eqref{Eq:Grad1}.
\begin{proposition} \label{prop:strongbound}
Assume the objective function $f$ is $\mu$-uniformly convex with respect to $h$ \eqref{eq:unif} and $h$ is $\sigma$-strongly convex. In addition, assume $f$ is $(1/\epsilon)$-smooth.
Using the sequences~\eqref{Eq:Coupling1} and \eqref{Eq:ZSeq1}, the following bound holds: 
\begin{align}\label{eq:errrrr}
\frac{E_{k+1} - E_k}{\delta} \leq  \varepsilon_{k+1},
\end{align}
where the error term has the following form: 
\begin{align*}
\varepsilon_{k+1} &= 
%
\frac{ A_{k+1} }{\delta}(f(y_{k+1}) - f(x_k)) +\frac{A_{k+1}}{\delta}\left( \frac{\tau_k}{2\epsilon}- \frac{\sigma\mu}{2\tau_k}\right)\|x_k - y_k\|^2  - \frac{A_{k+1} \mu\sigma}{2\delta}\|x_k - y_k\|^2  \\
&\quad + \frac{\alpha_k}{\delta}\langle \nabla f(x_k), y_k - x_k\rangle  +  \frac{A_{k+1}\sigma \mu}{2\delta}\|\tau_k(\nabla h(x_k) - \nabla h(z_k) - \frac{1}{\mu}\nabla f(x_k))\|^2. \notag
\end{align*}
When $h$ is Euclidean, the error simplifies to the following form
\begin{align*}
\varepsilon_{k+1} &=  \frac{A_{k+1}}{\delta}\left( f(y_{k+1}) - f(x_k) + \frac{\tau_k^2}{2\mu}\|\nabla f(x_k)\|^{2} +\left(\frac{\tau_k}{2\epsilon} -\frac{\mu }{2\tau_k} \right)\|x_k - y_k\|^2 \right). 
\end{align*}
\end{proposition}
We present a proof of Proposition~\ref{prop:strongbound} in Appendix~\ref{App:strongbound_proof}. 
 The result for accelerated gradient descent can be summed up in the following corollary, which is a consequence of Propositions~\ref{eq:prop1} and~\ref{prop:strongbound}.
 \begin{corollary}\label{prop:gradstr}
Using the gradient step,
\begin{align*}
\mathcal{G}(x) = x_{k}- \epsilon \nabla f(x_{k}),
\end{align*}
 for update~\eqref{Eq:Grad1} results in an error which scales as
\begin{align*}
\varepsilon_{k+1} &=\frac{A_{k+1}}{\delta}\left(\frac{\tau_k^2}{2\mu} - \frac{\epsilon}{2}  \right)\|\nabla f(x_k)\|^{2} +\frac{A_{k+1}}{\delta}\left(\frac{\tau_k}{2\epsilon} -\frac{\mu }{2\tau_k} \right)\|x_k - y_k\|^2,
\end{align*}
when $h$ is Euclidean.
\end{corollary}
The parameter choice $\tau_k \leq \sqrt{\mu\epsilon}= 1/\sqrt{\kappa}$ ensures the error is non-positive. With this choice, we obtain a linear $O(e^{-\sqrt{\mu \epsilon}k})=O(e^{-k/\sqrt{\kappa}}) $ convergence rate. Again, if we take the discretization step to scale according to the smoothness as $\delta = \sqrt{\epsilon}$, then both $\|x_k - y_k\| = O(\sqrt{\epsilon})$ and $\varepsilon_k = O(\sqrt{\epsilon})$, so we recover the dynamics~\eqref{eq:euc} and the continuous Lyapunov argument $\dot{\mathcal{E}}_t \leq 0$ in the limit $\sqrt{\epsilon} \rightarrow 0$.
\subsubsection{\bf Quasi-monotone method}
We end this section by studying a family of algorithms which can be thought of as a variation of the implicit Euler scheme applied to~\eqref{Eq:ELH1} and~\eqref{Eq:ELH1},
\begin{subequations}\label{eq:str1}
\begin{align}
x_{k+1} &= \frac{\tau_{k}}{1+\tau_{k}}z_k + \frac{1}{1+\tau_{k}} x_k\label{eq:strx}\\
\nabla h(z_{k+1}) &= \nabla h(z_k) + \tau_{k}\left(\nabla h(x_{k+1}) - \nabla h(z_{k+1}) - (1/\mu)\nabla f(x_{k+1})\right),\label{eq:strz}
\end{align}
\end{subequations}
where $\tau_k = \frac{A_{k+1} - A_k}{A_{k}}:=\frac{\alpha_k}{A_{k}} $. In discretization~\eqref{eq:strx}, the state $z_{k+1}$ has been replaced by the state $z_k$. 
When $h$ is Euclidean, we can write~\eqref{eq:strz} as the following update:
\begin{align*}
z_{k+1} = \arg \min_{z\in \X} \left\{ \langle  \nabla f(x_{k+1}), z\rangle + \frac{\mu}{2\tau_k}\left\|z - \tilde z_{k+1}\right\|^2\right\}.
\end{align*}
where $\tilde z_{k+1} = \frac{z_k + \tau_kx_{k+1}}{1+\tau_k}$. The update~\eqref{eq:strz} involves optimizing a linear approximation to the function regularized by a weighted combination of Bregman divergences. This yields the result summarized in the following proposition. 
\begin{proposition}\label{prop:quasi-strong}
 Assume $f$ is $\mu$-strongly convex with respect to $h$ and $h$ is $\sigma$-strongly convex. The following error bound:
\begin{align*}
\frac{E_{k+1} - E_k}{\delta} \leq  \varepsilon_{k+1},
\end{align*}
 can be shown for algorithm~\eqref{eq:str1} using Lyapunov function~\eqref{eq:lyapddd}, where the error scales as 
\begin{align}\label{eq:quas_bound2}
\varepsilon_{k+1} = \frac{ A_k\tau_{k}^2}{2\mu\sigma\delta} \|\nabla f(x_{k+1})\|^{2}.
\end{align}
\end{proposition}
No smoothness assumptions on $f$ and $h$ are needed to show this bound, and we can replace all the gradients with subgradients. If we assume that all the subgradients of $f$ are upper bounded in norm, then optimizing this bound results in an $f(x_k) - f(x^\ast) \leq O(1/k)$ convergence rate for the function value, which is optimal for subgradient methods designed for strongly convex functions.\footnote{In particular, this rate is achieved by taking $\tau_k = \frac{2}{k+2}$.}

\subsection{Frank-Wolfe algorithms}
\label{Sec:Frank}
In this section we describe how Frank-Wolfe algorithms can, in a sense, be considered as discrete-time mappings of dynamics which satisfy the conditions, 
 \begin{subequations}\label{Eq:Frankd}
 \begin{align}
 Z_t &= X_t + \dot \beta_t^{-1} \dot X_t,\label{Eq:Frank2}\\
0 \leq \langle &\nabla f(X_t), x-  Z_t\rangle,\quad \forall x \in \X. \label{Eq:Frank1} 
\end{align}
\end{subequations}
These dynamics are not guaranteed to exist; however, they are remarkably similar to the dynamics~\eqref{eq:el1}, where instead of using the Bregman divergence to ensure nonnegativity of the variational inequality $0\leq  \dot \beta_t e^{\beta_t}\langle \nabla f(X_t), x- Z_t\rangle$, we simply assume \eqref{Eq:Frank1} holds on the domain $\mathcal{X}$. We summarize the usefulness of dynamics~\eqref{Eq:Frankd} in the following proposition.
\begin{proposition} \label{prop:franklyap} Assume $f$ is convex and the ideal scaling \eqref{Eq:IdeScaBet} holds. The following function:
\begin{equation}\label{Eq:FrankE}
\mathcal{E}_t= e^{\beta_t} (f(X_t) - f(x)),
\end{equation} 
is a Lyapunov function for the dynamics which satisfies~\eqref{Eq:Frankd}. We can therefore conclude an $O(e^{-\beta_t})$ convergence rate of dynamics~\eqref{Eq:Frankd} to the minimizer of the function. 
\end{proposition}
The proof of this Proposition is in Appendix~\ref{App:franklyap}. Here, we will analyze two Frank-Wolfe algorithms that arise from dynamics~\eqref{Eq:Frankd}. 
Applying the backward-Euler scheme to~\eqref{Eq:Frank2} and~\eqref{Eq:Frank1}, with the same approximations, $\frac{d}{dt} X_t =\frac{x_{k+1} - x_k}{\delta}$, $\frac{d}{dt} e^{\beta_t} = \frac{A_{k+1} - A_k}{\delta}$, and denoting  $\tau_k = \frac{A_{k+1} - A_k}{A_{k+1}}$, we obtain the variational conditions for the following algorithm:
\begin{subequations}\label{eq:Frank}
\begin{align}
z_k &= \arg \min_{z \in \X}\,\,  \langle \nabla f(x_k), z\rangle, \label{Eq:XSeqFrank1}\\
x_{k+1} &= \tau_k z_k + (1- \tau_k) x_k.\label{Eq:ZSeqFrank1}
\end{align}
\end{subequations}
Update~\eqref{Eq:XSeqFrank1} requires the assumptions that $\mathcal{X}$ be convex and compact; under this assumption, \eqref{Eq:XSeqFrank1} satisfies
\begin{align*}
0 \leq \langle \nabla f(x_k), x - z_k \rangle, \forall x \in \mathcal{X},
\end{align*}
consistent with \eqref{Eq:Frank1}. The following proposition describes how a discretization of~\eqref{Eq:FrankE} can be used to analyze the behavior of algorithm~\eqref{eq:Frank}. 
\begin{proposition}\label{eq:propFrank}
Assume $f$ is convex and $\mathcal{X}$ is convex and compact. If f is $(1/\epsilon)$-smooth, using the Lyapunov function,
\begin{align}\label{eq:lyapfrank}
E_k = A_k (f(x_k) - f(x)),
\end{align}
we obtain the error bound, 
\begin{align*}
\frac{E_{k+1} - E_k}{\delta} \leq \varepsilon_{k+1},
\end{align*}
where the error for algorithm~\eqref{eq:Frank} scales as
\begin{align}\label{eq:bound1}
\varepsilon_{k+1} =\frac{A_{k+1} \tau_k^2}{2\epsilon\delta} \|z_k - x_k\|^2.
\end{align}
 If instead we assume $f$ has $(\epsilon,\nu)$-H\"older-continuous gradients~\eqref{eq:Hold}, the error in algorithm~\eqref{eq:Frank} now scales as
\begin{align}\label{eq:bound2}
\varepsilon_{k+1} = \frac{A_{k+1} \tau_k^{1+\nu}}{(1+\nu)\epsilon\delta} \|z_k - x_k\|^{1+\nu}.
\end{align}
\end{proposition}
Taking $x = x^\ast$ we infer the convergence rates $O(1/\epsilon k)$ and $O(1/\epsilon k^{\nu})$, respectively. We provide a proof of Proposition~\ref{eq:propFrank} in Appendix~\ref{App:prop_Frank}.
 \section{Equivalence to Estimate Sequences}
 \label{Sec:Est}
 In this section, we connect our Lyapunov framework directly to estimate sequences. We derive continuous-time estimate sequences directly from our Lyapunov function and demonstrate how these two techniques are equivalent.

\subsection{Estimate sequences}
We provide a brief review of the technique of estimate sequences~\cite{Nesterov04}.
We begin with the following definition.
\begin{definition}\cite[2.2.1]{Nesterov04} A pair of sequences $\{\phi_k(x)\}_{k=1}^\infty$ and $\{A_k\}_{k=0}^\infty$ $A_k \geq1$ is called an {\em estimate sequence} of function $f(x)$ if 
\[ 
A_k^{-1} \rightarrow 0,
\]
and, for any $x \in \R^n$ and for all $k \geq 0$, we have 
\begin{equation}\label{Eq:Ineq1}
\phi_k(x) \leq \Big(1 - A_k^{-1}\Big)f(x) + A_k^{-1}\phi_0(x).
\end{equation}
\end{definition}
\noindent The following lemma, due to Nesterov, explains why estimate sequences are useful.
%
%
\begin{lemma}\label{Lem:Nest2.2.1} \cite[2.2.1]{Nesterov04}
If for some sequence $\{x_k\}_{k\geq0}$ we have 
\begin{equation}\label{Eq:Seq}
f(x_k) \leq \phi_k^\ast \equiv \min_{x \in \X} \phi_k(x),
\end{equation}
then $f(x_k) - f(x^\ast) \leq A_k^{-1} [\phi_0(x^\ast) - f(x^\ast)]$. 
\end{lemma}
\begin{proof}The proof is straightforward:
\begin{align*}
f(x_k) \overset{\eqref{Eq:Seq}}{\leq} \phi_k^\ast \equiv \min_{x \in \X} \phi_k(x) &\overset{\eqref{Eq:Ineq1}}{\leq}\min_{x \in\X}\left[\Big(1 - A_k^{-1}\Big)f(x) + A_k^{-1}\phi_0(x)\right]\leq \Big(1 - A_k^{-1}\Big)f(x^\ast) + A_k^{-1}\phi_0(x^\ast).
\end{align*}
Rearranging gives the desired inequality.
\end{proof}
Notice that this definition is not constructive. Finding sequences which satisfy these conditions is a non-trivial task. The next proposition, formalized by Baes in \cite{Baes09} as an  extension of Nesterov's Lemma 2.2.2 \cite{Nesterov04}, provides guidance for constructing estimate sequences. This construction is used in \cite{Nesterov04, Nesterov05, Nesterov08, Baes09, Nesterov15, NesterovCond15}, and is, to the best of our knowledge, 
the only known formal way to construct an estimate sequence.  We will see below that this particular class of estimate sequences can be turned into our Lyapunov functions with a few algebraic manipulations (and vice versa).
\begin{proposition}\cite[2.2]{Baes09}\label{Eq:Prop}
 Let $\phi_0: \X \rightarrow \R$ be a convex function such that $\min_{x \in \X} \phi_0(x)\geq f^\ast$. Suppose also that we have a sequence $\{f_k\}_{k\geq 0}$ of functions from $\X$ to $\R$ that underestimates $f$:
\begin{align}\label{Eq:Under}
f_k(x) \leq f(x) \quad \text{ for all $x \in \X$ and all $k \geq 0$}.
\end{align}
Define recursively $A_0 = 1$, $\tau_k = \frac{A_{k+1} - A_k}{A_{k+1}} := \frac{\alpha_{k}}{A_k}$, 
and 
\begin{equation}\label{Eq:Est}
\phi_{k+1}(x) := (1 - \tau_k) \phi_k(x) + \tau_k f_k(x) = A_{k+1}^{-1} \left(A_0\phi_0(x) + \sum_{i=0}^k a_i f_i(x)\right),
\end{equation}
for all $k\geq 0$. Then $\left(\{\phi_k\}_{k\geq0}, \{A_k\}_{k\geq0}\right)$ is an estimate sequence.
\end{proposition}

\noindent From \eqref{Eq:Seq} and \eqref{Eq:Est}, we observe that the following  invariant:
\begin{align}\label{Eq:EstSeqInv}
A_{k+1} f(x_{k+1}) &\leq \min_x A_{k+1} \phi_{k+1}(x) = \min_x  \sum_{i=0}^k \alpha_i f_i(x) +  A_0\phi_0(x),
\end{align}
is maintained. In \cite{Nesterov15, NesterovCond15}, this technique was extended to incorporate an error term $\{\tilde \varepsilon_k\}_{k=1}^\infty$,
\begin{align*}
\phi_{k+1}(x) - A_{k+1}^{-1}\tilde \varepsilon_{k+1} &:= (1 - \tau_k) \Big(\phi_k(x) - A_k^{-1}\tilde \varepsilon_{k}\Big) + \tau_k f_k(x) = A_{k+1}^{-1} \Big(A_0(\phi_0(x) - \tilde \varepsilon_0) + \sum_{i=0}^k a_i f_i(x)\Big), 
\end{align*}
where $\varepsilon_k \geq 0, \forall k$.
Rearranging, we have the following bound: 
\begin{align*}
A_{k+1} f(x_{k+1}) &\leq \min_x A_{k+1}\phi_{k+1}(x)= \min_x  \sum_{i=0}^k \alpha_i f_i(x) +  A_0\Big(\phi_0(x) - A_0^{-1}\tilde \varepsilon_0\Big) + \tilde \varepsilon_{k+1}.
\end{align*}
Notice that an argument analogous to that of Lemma~\ref{Lem:Nest2.2.1} holds:
\begin{align*}
A_{k+1} f(x_{k+1}) \leq \sum_{i=0}^k \alpha_i f_i(x^\ast) + A_0 (\phi_0(x^\ast) - \tilde \varepsilon_0)  + \tilde \varepsilon_{k+1}\notag &\overset{\eqref{Eq:Under}}{\leq} \sum_{i=0}^k \alpha_i f(x^\ast) +  A_0\phi_0(x^\ast)  + \tilde \varepsilon_{k+1}\notag\\ &= A_{k+1} f(x^\ast) +  A_0\phi_0(x^\ast) + \tilde \varepsilon_{k+1}.
\end{align*}
 Rearranging, we obtain the desired bound,
\begin{align*}
f(x_{k+1}) - f(x^\ast) \leq\frac{A_0\phi_0(x^\ast)  + \tilde \varepsilon_{k+1}}{A_{k+1}}.
\end{align*}
Thus, we simply need to choose our sequences $\{A_k, \phi_k, \tilde \varepsilon_k\}_{k=1}^\infty$ to ensure $\tilde \varepsilon_{k+1}/ A_{k+1} \rightarrow 0$.  The following table illustrates the choices of $\phi_k(x)$ and $\tilde \varepsilon_k$ for the four methods discussed earlier. 
\begin{center}
\begin{table}[h!]\label{Eq:Table}
\tabcolsep=0.2cm
\begin{center}
\begin{tabular}{|c|c|c|c|}
\hline 
Algorithm & $f_i(x)$ &  $\phi_k(x)$ & $\tilde \varepsilon_{k+1}$ \\ \hline
$\substack{\text{Quasi-Monotone Subgradient Method}}$ & linear & $\frac{1}{A_k}D_h(x, z_k) + f(x_k)$ & $\frac{1}{2}\sum_{i=1}^{k+1} \frac{(A_i - A_{i-1})^2}{2}G^2$  \\ \hline
$\substack{\text{Accelerated Gradient Method} \\\text{ (Weakly Convex)}} $& linear  &$\frac{1}{A_k}D_h(x, z_k) +  f(x_k) $ & 0 \\  \hline
$\substack{\text{Accelerated Gradient Method}\\\text{ (Strongly Convex)}} $  & quadratic &$  f(x_k)  + \frac{\mu}{2}\|x - z_k\|^2$ &  0 \\ \hline
$\substack{\text{Conditional Gradient Method }}$ & linear  & $f(x_k)$ & $ \frac{1}{2\epsilon}\sum_{i=1}^{k+1}\frac {(A_{i} - A_{i-1})^2}{A_{i}} diam(\X)^2$\\
\hline
\end{tabular}
\end{center}
\caption{Choices of estimate sequences for various algorithms}
\label{table:Table}
\end{table}
\end{center}
In Table~\ref{table:Table} ``linear'' is defined as 
$f_i(x) = f(x_i) + \langle \nabla f(x_i), x - x_i\rangle,$
and ``quadratic'' is defined as
$f_i(x) = f(x_i) + \langle \nabla f(x_i), x - x_i\rangle + \frac{\mu}{2}\|x - x_i\|^2.$ The estimate-sequence argument is inductive; one must know the three sequences $\{\varepsilon_{k}, A_k, \phi_k(x)\}$ a priori in order to check the invariants hold. This aspect 
of the estimate-sequence technique has made it hard to discern its structure and 
scope. 

\subsection{Equivalence to  Lyapunov functions}
We now demonstrate an equivalence between these two frameworks.  The continuous-time view shows that the errors in both the Lyapunov function and estimate sequences are due to discretization errors. We demonstrate how this works for accelerated methods, and defer the proofs for the other algorithms discussed earlier in the paper to Appendix~\ref{App:EstSeq}.
\paragraph{Equivalence in discrete time.}
The discrete-time estimate sequence~\eqref{Eq:Est} for accelerated gradient descent can be written:
\begin{align*}
\phi_{k+1}(x) &:= f(x_{k+1}) + A_{k+1}^{-1}D_h(x, z_{k+1}) \\
& \overset{\eqref{Eq:Est}}{=} (1 - \tau_k) \phi_k(x)  + \tau_k f_k(x) \\
&\overset{\text{Table}~\ref{table:Table}}{=} \Big(1 - A_{k+1}^{-1}\alpha_k \Big)\Big(f(x_k) +A_k^{-1} D_h(x, z_k)\Big) +A_{k+1}^{-1} \alpha_k f_k(x).
\end{align*}
Multiplying through by $A_{k+1}$, we have the following argument, which follows 
directly from our definitions:
\begin{align*}
A_{k+1} f(x_{k+1}) + D_h(x, z_{k+1}) &= (A_{k+1} - \alpha_k) \Big(f(x_k) + A_k^{-1}D_h(x, z_k)\Big) + \alpha_k f_k(x)\\
&=  A_k \Big(f(x_k) + A_k^{-1}D_h(x, z_k) \Big)+ (A_{k+1} - A_k) f_k(x)\\
& \leq  A_k f(x_k) + D_h(x, z_k)+ (A_{k+1} - A_k) f(x).
\end{align*}
The last inequality follows from definition~\eqref{Eq:Under}. Rearranging, we obtain the inequality $E_{k+1} \leq E_k$ for our Lyapunov function~\eqref{eq:lyap}. 
Going the other direction, from our Lyapunov analysis we can derive the following bound:
\begin{align}
E_k &\leq E_0\notag\\
A_{k}(f(x_{k}) - f(x)) + D_h(x, z_{k}) &\leq A_0 (f(x_0) - f(x)) + D_h(x, z_0) \notag\\
A_{k}\Big(f(x_{k}) - A_k^{-1}D_h(x, z_{k})\Big) &\leq (A_k -  A_0) f(x) + A_0\Big(f(x_0)+ A_0^{-1}  D_h(x^\ast, z_0)\Big) \notag \\
A_k \phi_k(x) &\leq (A_k -A_0) f(x) + A_0 \phi_0(x). \label{Eq:Sequence}
\end{align}
Rearranging, we obtain the estimate sequence~\eqref{Eq:Ineq1}, with $A_0 = 1$:
\begin{align*}
\phi_k(x) &\leq \Big(1 -A_k^{-1}A_0\Big) f(x) + A_k^{-1}A_0 \phi_0(x) =\Big(1 - A_k^{-1}\Big)f(x) + A_k^{-1}\phi_0(x).
\end{align*}
Writing $\mathcal{E}_t \leq \mathcal{E}_0$, one can simply rearrange terms to extract an estimate sequence:
\begin{align*}
f(X_t) + e^{-\beta_t}D_h\left(x, Z_t\right) &\leq \Big(1- e^{-\beta_t}e^{\beta_0}\Big) f(x^\ast) +  e^{-\beta_t}e^{\beta_0}\Big(f(X_0)  + e^{-\beta_0}D_h\left(x, Z_0\right)\Big).
\end{align*}
Comparing this to~\eqref{Eq:Sequence}, matching terms allows us to extract the 
continuous-time estimate sequence
$\{\phi_t(x) , e^{\beta_t}\}$, where $\phi_t(x) = f(X_t) + e^{-\beta_t}D_h(x, Z_t)$.

 
\section{Further Observations}
The dynamical perspective can be extended to the derivation and analysis 
of a range of other methods. In this section, we provide sketches of some
of these analyses, providing a detailed treatment in Appendix~\ref{Sec:AddObv}.
\paragraph{Proximal methods.}
Methods for minimizing the composite of two convex functions, 
$\varphi(x) = f(x) + \psi(x)$, were introduced by Nesterov~\cite{Nesterov13} 
and studied by Beck and Teboulle~\cite{BeckTeboulle09}, Tseng~\cite{Tseng08} 
and several others. In Appendix~\ref{Sec:Prox}, we present a dynamical 
perspective on these methods and show how to recover their convergence
theory via the Lyapunov functions presented in this paper. 
\paragraph{Stochastic methods.}
We sketch a high-level view of algorithms which use stochastic estimates 
of gradients, and provide a more detailed analysis in Appendix~\ref{App:Stoch}. 
Our scope is a Lyapunov-based analysis of four algorithms---stochastic 
mirror descent with momentum, accelerated (proximal) coordinate 
descent~\cite{Zeyuan16,NestStich17,Tu17,AccCordProx1,AccCordProx2}, 
accelerated stochastic variance reduction (SVRG)~\cite{Zhu17}, and 
accelerated stochastic composite methods~\cite{Lan12}.  We study these
methods under two smoothness settings and present proofs for several 
explicit methods. Broadly, we consider algorithms~\eqref{eq:agd}, 
\eqref{eq:str1} and~\eqref{Eq:StrongConv21}, where stochastic gradients 
are used instead of full gradients. For these methods, we show the bound 
$\mathbb{E}[E_{k+1}]- E_{k} \leq \mathbb{E}[\varepsilon_{k+1}]$
for Lyapunov function~\eqref{eq:nlyap} and
$
\mathbb{E}[E_{k+1}]- E_{k} \leq -\tau_k E_k + \mathbb{E}[\varepsilon_{k+1}]
$
for Lyapunov function~\eqref{eq:lyapdd}, where the expectation is taken 
conditioned on the previous state. By summing, we obtain convergence rates 
for the aforementioned algorithms, provided the sequence $(A_i)_{i=1}^\infty$ 
is chosen so that $\mathbb{E}[\sum_{i=1}^\infty \varepsilon_i ]< \infty$.

\section{Discussion}
The main contributions in this paper are twofold: We have presented a 
unified analysis of a wide variety of algorithms using three Lyapunov 
functions--\eqref{eq:lyap},~\eqref{eq:lyapdd} and \eqref{eq:lyapfrank},
and we have demonstrated the equivalence between Lyapunov functions and 
estimate sequences, under the formalization of the latter due to Baes~\cite{Baes09}.
More generally, we have provided a dynamical-systems perspective that 
builds on Polyak's early intuitions, and elucidates connections between 
discrete-time algorithms and continuous-time, dissipative second-order 
dynamics.  We believe that the dynamical perspective renders the design 
and analysis of accelerated algorithms for optimization particularly 
transparent, and we also note in passing that Lyapunov analyses for 
non-accelerated gradient-based methods, such as mirror descent and 
natural gradient descent, can be readily derived from analyses of 
gradient-flow dynamics.  

We close with a brief discussion of some possible directions for future 
work.  First, we remark that requiring a continuous-time Lyapunov 
function to remain a Lyapunov function in discrete time places significant 
constraints on which ODE solvers can be used.  In this paper, we show 
that we can derive new algorithms using a restricted set of ODE techniques 
(several of which are nonstandard) but it remains to be seen if other 
methods can be applied in this setting.  Techniques such as the midpoint 
method and Runge Kutta provide more accurate solutions of ODEs than Euler 
methods~\cite{Butcher20001}.  Is it possible to analyze such techniques 
as optimization methods?  We expect that these methods do not achieve 
better asymptotic convergence rates, but may inherit additional favorable 
properties.  Determining the advantages of such schemes could provide 
more robust optimization techniques in certain scenarios.  In a similar
vein, it would be of interest to analyze the symplectic integrators
studied by~\cite{BetancourtJordanWilson} within our Lyapunov framework.

Several restart schemes have been suggested for the strongly convex setting based on the momentum dynamics~\eqref{eq:el1}. In many settings,  while the Lipschitz parameter can be estimated using backtracking line-search, the strong convexity parameter is often hard---if not impossible---to estimate~\cite{SuBoydCandes}.  Therefore, many authors~\cite{ODonoghue15,SuBoydCandes,Krichene15} have developed heuristics to empirically speed up the convergence rate of the ODE (or discrete-time algorithm), based on model misspecification. In particular, both Su, Boyd, and Candes~\cite{SuBoydCandes} and Krichene, Bayen and Bartlett~\cite{Krichene15} develop restart schemes designed for the strongly convex setting based on the momentum dynamics~\eqref{eq:el1}. Our analysis suggests that restart schemes based on the dynamics~\eqref{eq:el2} might lead to better results. 


Earlier work by Drori and Teboulle~\cite{DroriTeboulle13}, Kim and Fessler~\cite{Kim2016}, Taylor \emph{et al}~\cite{Taylor2016}, and Lessard~\emph{et al}~\cite{Lessard14} have shown that optimization algorithms can be analyzed by solving convex programming problems.  In particular, Lessard~\emph{et al} show that Lyapunov-like potential functions called \emph{integral quadratic constraints} can be found by solving a constant-sized semidefinite programming problem. It would be interesting to see if these results can be adapted to directly search for Lyapunov functions like those studied in this paper.  This would provide a method to automate the analysis of new techniques, possibly moving beyond momentum methods to novel families of optimization techniques.

\section*{Acknowledgements}
We would like to give special thanks to Andre Wibisono as well as Orianna Demassi 
and Stephen Tu  for the many helpful discussions involving this paper.   
ACW was supported by an NSF Graduate Research Fellowship. This work was 
supported in part by the Army Research Office under grant number W911NF-17-1-0304
and by the Mathematical Data Science program of the Office of Naval Research.

\bibliographystyle{siamplain}
\bibliography{refs.bib}
\clearpage
\appendix
\section{Dynamics}
\label{App:sec_dyn}
\subsection[Proof of Proposition]{Proof of Proposition~\ref{Prop:EL}}
\label{App:Prop_EL}
We compute the Euler-Lagrange equation for the second Bregman Lagrangian~\eqref{eq:breg2}. 
Denote $z = x + e^{-\alpha_t} \dot x$. The partial  
derivatives of the Bregman Lagrangian can be written,
\begin{align*}
\frac{\partial\mathcal{L}}{\partial v} (X_t, \dot X_t, t)&= \mu  e^{ \beta_t + \gamma_t} \left(\nabla h(Z_t) - \nabla h(X_t)\right)  \\
\frac{\partial\mathcal{L}}{\partial x}(X_t, \dot X_t, t) &= \mu e^{\alpha_t} \frac{\partial\mathcal{L}}{\partial v} (X_t, \dot X_t, t) - \mu e^{ \beta_t + \gamma_t}\frac{d}{dt}\nabla h(X_t) - e^{\alpha_t + \beta_t + \gamma_t} \nabla f(X_t).
\end{align*}
We also compute the time derivative of the momentum  $p = \frac{\partial\mathcal{L}}{\partial v} (X_t, \dot X_t, t)$,
\begin{align*}
\frac{d}{dt}\frac{\partial\mathcal{L}}{\partial v}(X_t, \dot X_t, t) &= (\dot \beta_t + \dot \gamma_t) \frac{\partial\mathcal{L}}{\partial v} (X_t, \dot X_t, t)+ \mu  e^{ \beta_t + \gamma_t}\frac{d}{dt}\nabla h(Z_t)  - \mu  e^{\beta_t + \gamma_t} \frac{d}{dt}\nabla h(X_t). 
\end{align*} 
The terms involving $\frac{d}{dt}\nabla h(X)$ cancel and the terms involving the momentum will simplify under the scaling condition~\eqref{Eq:IdeScaGam} when computing the Euler-Lagrange equation $ \frac{\partial\mathcal{L}}{\partial x}(X_t, \dot X_t, t) = \frac{d}{dt}\frac{\partial\mathcal{L}}{\partial v}(X_t, \dot X_t, t)$. Compactly, the Euler-Lagrange equation can be written 
\begin{align*}
\frac{d}{dt} \mu \nabla h(Z_t)  =  - \dot \beta_t\mu \left(\nabla h(Z_t) - \nabla h(X_t)\right) - e^{\alpha_t} \nabla f(x).
\end{align*}

\paragraph{Remark.}
It is interesting to compare with the partial derivatives of the 
first Bregman Lagrangian~\eqref{eq:breg},
\begin{subequations}
\begin{align*}
\frac{\partial\mathcal{L}}{\partial v} (X_t, \dot X_t, t)&= e^{ \gamma_t} \left(\nabla h(Z_t) - \nabla h(X_t)\right) \\
\frac{\partial\mathcal{L}}{\partial x}(X_t, \dot X_t, t) &=  e^{\alpha_t} \frac{\partial\mathcal{L}}{\partial v} (X_t, \dot X_t, t) -  e^{\gamma_t}\frac{d}{dt}\nabla h(X_t) - e^{\alpha_t + \beta_t + \gamma_t} \nabla f(X_t),
\end{align*}
\end{subequations}
as well as the derivative of the momentum,
\begin{align*}
\frac{d}{dt}\frac{\partial\mathcal{L}}{\partial v} (X_t, \dot X_t, t)&= \dot \gamma_t \frac{\partial\mathcal{L}}{\partial v} (X_t, \dot X_t, t) + e^{ \gamma_t}\frac{d}{dt}\nabla h(Z_t) -  e^{ \gamma_t}\frac{d}{dt}\nabla h(X_t). 
\end{align*}
For Lagrangian~\eqref{eq:breg}, not only do the terms involving $\frac{d}{dt}\nabla h(X)$ cancel when computing the Euler-Lagrange equation, but the ideal scaling will also force the terms involving the momentum to cancel as well. 
\subsection{Deriving the Lyapunov functions}
\label{App:deriv_lyap}

\subsubsection[Proof of Proposition]{Proof of Proposition~\ref{Prop:WeakLyap}} 

We demonstrate how to derive the Lyapunov function~\eqref{eq:lyap} for the momentum dynamics~\eqref{eq:el1}; this derivation is similar in spirit to the Lyapunov analysis of mirror descent by Nemirovski and Yudin.  Denote $Z_t = X_t + e^{-\alpha_t} \dot X_t$. We have:
\begin{subequations}\label{eq:iden}
\begin{align}
 \frac{d}{dt}D_h\left(x, Z_t\right)  &= \frac{d}{dt}\left( h(x) - h(Z_t) - \langle \nabla h(Z_t), x- Z_t\rangle\right)  \notag\\
& = - \langle \nabla h(Z_t), \dot Z_t\rangle -\left\langle \frac{d}{dt} \nabla h(Z_t),x - Z_t\right\rangle  + \langle \nabla h(Z_t), \dot Z_t\rangle \notag\\ 
 &= -\left\langle \frac{d}{dt} \nabla h\left(Z_t\right), x - Z_t\right\rangle. \, \notag
 \end{align}
 \end{subequations}
Using this identity, we obtain the following argument:
\begin{subequations}\label{Eq:LyapAnal}
\begin{align}
 \frac{d}{dt}D_h\left(x, Z_t\right)  
 &= -\left\langle \frac{d}{dt} \nabla h\left(Z_t\right), x - Z_t\right\rangle \, \notag\\
&=    e^{\alpha_t +\beta_t}\left\langle\nabla f(X_t), x  - X_{t} - e^{-\alpha_t} \dot X_{t} \right\rangle  \, \label{eq:apply-dynamics}   \\
& =   e^{\alpha_t + \beta_t} \langle \nabla f(X_t), x - X_t \rangle\, - e^{ \beta_t} \langle \nabla f(X_t), \dot X_t\rangle \, \notag   \\
& = e^{\alpha_t+ \beta_t} \langle \nabla f(X_t), x - X_t \rangle\,  -   \frac{d}{dt}\left(e^{ \beta_t} f(X_t)\right)\,   + \dot \beta_t e^{\beta_t}  f(X_t)  \notag\\
&= \dot \beta_t e^{\beta_t} [f(X_t)+ \langle \nabla f(X_t), x - X_t \rangle] -\frac{d}{dt}\left(e^{ \beta_t} f(X_t)\right) +e^{\beta_t} \Big(e^{\alpha_t} - \dot \beta_t\Big)\langle \nabla f(X_t), x - X_t \rangle\, \notag  \\ 
& \leq  \dot \beta_t e^{\beta_t}f(x) - \frac{d}{dt}\left(e^{ \beta_t} f(X_t)\right)\, \label{eq:apply-convexity}\\
& = - \frac{d}{dt} \Big\{e^{\beta_t}\left(f(X_t) - f(x)\right)\Big\} \label{Eq:cond1}.
\end{align}
\end{subequations}
Here~\eqref{eq:apply-dynamics} uses the momentum dynamics
~\eqref{Eq:ELH} and~\eqref{Eq:ELZ}. 
The inequality~\eqref{eq:apply-convexity} follows from the convexity of $f$. If $\dot \beta_t = e^{\alpha_t}$, simply by rearranging terms and taking $x = x^\ast$, we have shown that the function~\eqref{eq:Lyap1} has nonpositive derivative for all $t$ and is hence a Lyapunov function for the family of momentum dynamics~\eqref{eq:el1}. If $\dot \beta_t \leq e^{\alpha_t}$, the Lyapunov function is only decreasing for $x = x^\ast$. 

\subsubsection[Proof of Proposition]{Proof of Proposition~\ref{prop:strLyap}}
\label{app:strLyapproof}
We demonstrate how to derive the Lyapunov function~\eqref{eq:Lyap2} for the momentum dynamics~\eqref{eq:el2}. 
Using the same identity~\eqref{eq:iden}, we have the following initial,
\begin{subequations}
\begin{align*}
\frac{d}{dt} \Big\{ e^{\beta_t} \mu D_h\left(x, Z_t\right)\Big\} &=  - e^{\beta_t}\mu \left\langle \frac{d}{dt} \nabla h\left(Z_t\right), x - Z_t\right\rangle + \mu\dot \beta_t e^{\beta_t}D_h\left(x, Z_t\right)\notag \\
& = \mu \dot \beta_t e^{\beta_t} \left[ \left\langle \nabla h\left(Z_t\right) - \nabla h(X_t),x - Z_t\right\rangle +D_h\left(x, Z_t\right)\right]  \notag\\
 & \quad + \dot \beta_t e^{\beta_t}\left\langle \nabla f(X_t), x - Z_t\right\rangle+ \left(e^{\alpha_t} - \dot \beta_t\right) \left\langle \nabla f(X_t), x - Z_t\right\rangle. \notag
 \end{align*}
 \end{subequations}
The Bregman three-point identity,
\begin{align}\label{eq:bregthree}
\langle \nabla h(Z_t) - \nabla h(X_t), x- Z_t\rangle + D_h(x, Z_t)   = D_h(x, X_t)-D_h(Z_t, X_t), 
\end{align}
will now be useful. Proceeding from the last line, we have
\begin{subequations}
\begin{align*}
\frac{d}{dt} \Big\{ e^{\beta_t} \mu D_h\left(x, Z_t\right)\Big\} & =  \dot \beta_t e^{\beta_t}\left[  \left\langle \nabla f(X_t), x - X_t\right\rangle  + \mu D_h(x, X_t)\right] - \mu\dot \beta_t e^{\beta_t} D_h\left( Z_t, X_t\right) \notag \\
  &\quad -  e^{\beta_t}  \left\langle \nabla f(X_t) ,\dot X_t\right\rangle   + \left(e^{\alpha_t} - \dot \beta_t\right) \left\langle \nabla f(X_t), x - X_t\right\rangle\notag\\
&\leq - \dot \beta_t e^{\beta_t}  (f(X_t) - f(x)) + \dot \beta_t e^{\beta_t}  f(X_t) -\frac{d}{dt} \left\{e^{\beta_t} f(X_t)\right\}\notag\\
&\quad   - \mu \dot \beta_t e^{\beta_t}  D_h\left( Z_t, X_t\right) + \left(e^{\alpha_t} - \dot \beta_t\right) \left\langle \nabla f(X_t), x - Z_t\right\rangle\notag\\
& \leq  -\frac{d}{dt} \Big\{ e^{\beta_t} (f(X_t) - f(x))\Big\}.\notag
\end{align*}
\end{subequations}
The first inequality follows from the $\mu$-uniform convexity of $f$ with respect to $h$.  The second inequality follows from nonnegativity of the Bregman divergence, and the ideal scaling condition~\eqref{Eq:IdeScaBet}, where we must take $x= x^\ast$ if $\dot \beta_t \leq e^{\alpha_t}$.
\section{Algorithms derived from dynamics~\eqref{eq:el1}}

%
\subsection[Proof of Proposition]{Proof of Proposition~\ref{Prop:WeakBound}}
\label{App:WeakBound}
We show the initial bounds~\eqref{eq:err} and \eqref{eq:err1}. We begin with algorithm~\eqref{eq:agd}:
\begin{align*}
E_{k+1} - E_k &= D_h(x, z_{k+1})  -  D_h(x, z_{k}) + A_{k+1}(f(y_{k+1}) - f(x)) - A_{k}(f(y_{k}) - f(x))\\
& = - \langle  \nabla h(z_{k+1}) - \nabla h(z_k), x - z_{k+1}\rangle - D_h(z_{k+1}, z_k) + A_{k+1}(f(y_{k+1}) - f(x))- A_{k}(f(y_{k}) - f(x))\\
& \overset{\eqref{eq:mir}}{=} \alpha_k \langle  \nabla f(x_{k+1}), x - z_{k+1}\rangle - D_h(z_{k+1}, z_k) + \alpha_k(f(x_{k+1}) - f(x)) + A_k (f(x_{k+1}) - f(y_k))
\\\quad &+ A_{k+1}(f(y_{k+1}) - f(x_{k+1}))\\
& \leq  \alpha_k \langle  \nabla f(x_{k+1}), x - z_{k}\rangle + \alpha_k\langle  \nabla f(x_{k+1}), z_k- z_{k+1}\rangle - \frac{\sigma}{p}\|z_{k+1} - z_k\|^p + \alpha_k(f(x_{k+1}) - f(x)) 
\\\quad &+ A_k (f(x_{k+1}) - f(y_k)) + A_{k+1}(f(y_{k+1}) - f(x_{k+1}))\\
& \leq  \alpha_k \langle  \nabla f(x_{k+1}), x - z_{k}\rangle  + A_k (f(x_{k+1}) - f(y_k)) + \alpha_k(f(x_{k+1}) - f(x))\\
&\quad + \frac{p-1}{p}\sigma^{-\frac{1}{p-1}}(A_{k+1} - A_k)^{\frac{p}{p-1}}\|\nabla f(x_{k+1})\|^{\frac{p}{p-1}}   + A_{k+1}(f(y_{k+1}) -f(x_{k+1})). 
\end{align*}
The first inequality follows from the $\sigma$-uniform convexity of $h$ with respect to the $p$-th power of the norm and the last inequality follows from the Fenchel Young inequality. If we continue with our argument, and plug in the identity~\eqref{eq:err}, it simply remains to use our second update~\eqref{eq:coup}:
\begin{align*}
E_{k+1} - E_k& \leq  \alpha_k \langle  \nabla f(x_{k+1}), x - z_{k}\rangle  + A_k (f(x_{k+1}) - f(y_k)) + \alpha_k(f(x_{k+1}) - f(x))\\
&\quad + \frac{p-1}{p}\sigma^{-\frac{1}{p-1}}(A_{k+1} - A_k)^{\frac{p}{p-1}}\|\nabla f(x_{k+1})\|^{\frac{p}{p-1}}   + A_{k+1}(f(y_{k+1}) -f(x_{k+1}))\\
&\leq  \alpha_k \langle  \nabla f(x_{k+1}), x - y_{k}\rangle + A_{k+1}\langle  \nabla f(x_{k+1}), y_k - x_{k+1}\rangle +A_k (f(x_{k+1}) - f(y_k)) \\\quad & + \alpha_k(f(x_{k+1}) - f(x))+\varepsilon_{k+1} \\
&=  \alpha_k(  f(x_{k+1}) - f(x) + \langle 
\nabla f(x_{k+1}), x - x_{k+1}\rangle) + A_k (f(x_{k+1}) - f(y_k)  + \langle \nabla f(x_{k+1}), y_{k} - x_{k+1}\rangle) \\
&\quad +   \varepsilon_{k+1}.
\end{align*}
From here, we can conclude $E_{k+1} - E_k \leq \varepsilon_k$ using the convexity of $f$. 

We now show the bound~\eqref{eq:err1} for algorithm~\eqref{eq:agd2} using a similar argument.
\begin{align*}
E_{k+1} - E_k &= D_h(x, z_{k+1})  -  D_h(x, z_{k}) + A_{k+1}(f(y_{k+1}) - f(x)) - A_{k}(f(y_{k}) - f(x))\\
& \overset{\eqref{eq:mir}}{=}  \alpha_k\langle  \nabla f(y_{k+1}), x - z_{k+1}\rangle - D_h(z_{k+1}, z_k) + \alpha_k (f(y_{k+1}) - f(x)) + A_k (f(y_{k+1}) - f(y_k))\\
& \leq  \alpha_k\langle  \nabla f(y_{k+1}), x - z_{k}\rangle + \alpha_k\langle  \nabla f(y_{k+1}), z_k- z_{k+1}\rangle - \frac{\sigma}{p}\|z_{k+1} - z_k\|^p
\\\quad &+ \alpha_k(f(y_{k+1}) - f(x)) + A_k (f(y_{k+1}) - f(y_{k})) \\
& \leq  \alpha_k\langle  \nabla f(y_{k+1}), x - z_{k}\rangle  + A_k (f(y_{k+1}) - f(y_k))+ \alpha_k(f(y_{k+1}) - f(x))\\\quad & - A_{k+1}\langle \nabla f(y_{k+1}), y_{k+1} - x_{k+1}\rangle  + \varepsilon_{k+1}.
\end{align*}
The first inequality follows from the uniform convexity of $h$ and the second uses the Fenchel Young inequality and definition~\eqref{eq:err1}. 
Using the second update~\eqref{eq:coup1}, we obtain our initial error bound:
\begin{align*}
E_{k+1} - E_k&\leq  \alpha_k\langle  \nabla f(y_{k+1}), x - y_{k}\rangle + A_k (f(y_{k+1}) - f(y_k))+ \alpha_k(f(y_{k+1}) - f(x)) \\ 
&\quad + A_{k+1}\langle  \nabla f(y_{k+1}), y_k - x_{k+1}\rangle- A_{k+1}\langle \nabla f(y_{k+1}), y_{k+1} - x_{k+1}\rangle+\varepsilon_{k+1} \\
&=  \alpha_k(  f(y_{k+1}) - f(x) + \langle 
\nabla f(y_{k+1}), x - y_{k+1}\rangle)   \\&\quad + A_k (f(y_{k+1}) - f(y_k) + \langle \nabla f(y_{k+1}), y_{k} - y_{k+1}\rangle)+
\varepsilon_{k+1}.
\end{align*}
The last line can be upper bounded by the error $\varepsilon_{k+1}$ using convexity of $f$. 

\subsection[Proof of Proposition]{Proof of Proposition~\ref{prop:Hold}}
\label{App:Hold}
A similar progress bound was proved in Wibisono, Wilson and Jordan~\cite[Lem 3.2]{Acceleration}. Note that $y = \mathcal{G}(x)$ satisfies the optimality condition 
\begin{align}\label{Eq:UpdateProof1}
\sum_{i=1}^{p-1} \frac{1}{(i-1)!} \nabla^i f(x) \, (y-x)^{i-1} + \frac{N}{\epsilon} \|y-x\|^{\tilde p-2} \, (y-x) = 0.
\end{align}
Furthermore, since $\nabla^{p-1} f$ is H\"older-continuous~\eqref{eq:Hold}, we have the following error bound on the $(p-2)$-nd order Taylor expansion of $\nabla f$,
\begin{multline}\label{Eq:UpdateProof2}
\left\|\nabla f(y) - \sum_{i=0}^{p-1} \frac{1}{(i-1)!} \nabla^i f(x)(y-x)^{i-1}\right\| = \left\|\int_0^1 [\nabla^{p-1} f(ty + (1- t)x) - \nabla^{p-1} f(x)](y - x)^{p-2}dt\right\| \\\le \frac{1}{\epsilon} \|y-x\|^{p-2+ \nu} \int_0^1 t^\nu  = \frac{1}{\epsilon} \|y-x\|^{\tilde p -1}.
\end{multline}
Substituting~\eqref{Eq:UpdateProof1} to~\eqref{Eq:UpdateProof2} and writing $r = \|y-x\|$, we obtain
\begin{align}\label{Eq:UpdateProof2a}
\left\|\nabla f(y) + \frac{Nr^{\tilde p-2}}{\epsilon} \, (y-x)\right\|_* \,\le\, \frac{r^{\tilde p-1}}{\epsilon}.
\end{align}
Now the argument proceeds as in~\cite{Acceleration}.
Squaring both sides, expanding, and rearranging the terms, we get the inequality
\begin{align}\label{Eq:UpdateProof3}
\langle \nabla f(y), x-y \rangle
\,\ge\, \frac{\epsilon}{2Nr^{\tilde p-2}} \|\nabla f(y)\|_*^2 + \frac{(N^2-1)r^{\tilde p}}{2N\epsilon}.
\end{align}
Note that if $\tilde p=2$, then the first term in~\eqref{Eq:UpdateProof3} already implies the desired bound~\eqref{eq:progcond}. Now assume $\tilde p \ge 3$. The right-hand side of~\eqref{Eq:UpdateProof3} is of the form $A/r^{\tilde p-2} + Br^{\tilde p}$, which is a convex function of $r > 0$ and minimized by $r^* = \left\{\frac{(\tilde p-2)}{\tilde p} \frac{A}{B} \right\}^{\frac{1}{2\tilde p-2}}$, yielding a minimum value of
\begin{align*}
\frac{A}{(r^*)^{ \tilde p-2}} + B(r^*)^p
\,=\, A^{\frac{\tilde p}{2\tilde p-2}} B^{\frac{\tilde p-2}{2\tilde p-2}} \left[\left(\frac{\tilde p}{\tilde p-2}\right)^{\frac{\tilde p-2}{2\tilde p-2}} + \left(\frac{\tilde p-2}{\tilde p}\right)^{\frac{\tilde p}{\tilde p-2}}\right]
\,\ge\, A^{\frac{\tilde p}{2\tilde p-2}} B^{\frac{\tilde p-2}{2\tilde p-2}}.
\end{align*}
Substituting the values $A = \frac{\epsilon}{2N} \|\nabla f(y)\|_*^2$ and $B = \frac{1}{2N\epsilon} (N^2-1)$ from~\eqref{Eq:UpdateProof3}, we obtain
\begin{multline*}
\langle \nabla f(y), x-y \rangle
\,\ge\, \left(\frac{\epsilon}{2N} \|\nabla f(y)\|_*^2\right)^{\frac{\tilde p}{2 \tilde p-2}} \left(\frac{1}{2N\epsilon} (N^2-1)\right)^{\frac{ \tilde p-2}{2\tilde p-2}}
= \frac{(N^2-1)^{\frac{\tilde p-2}{2 \tilde p-2}}}{2N} \epsilon^{\frac{1}{\tilde p-1}} \|\nabla f(y)\|_*^{\frac{\tilde p}{\tilde p-1}},
\end{multline*}
which proves the progress bound~\eqref{eq:progcond}.
\subsection{Proof of Universal Gradient Method}
\label{App:UnivGrad}
We present a convergence rate for higher-order gradient method $y_{k+1} = \mathcal{G}_{\epsilon, p, \nu,N}(x_{k+1})$ where $\mathcal{G}$ is given by~\eqref{eq:taylor3} and $f$ has $(\epsilon, \nu)$-H\"older-continuous gradients~\eqref{eq:Hold}. The proof is inspired by the proof of the rescaled gradient flow
$
\dot X_t = -\nabla f(X_t)/\|\nabla f(X_t)\|_\ast^{\frac{p-2}{p-1}},
$
outlined in~\cite[Appendix G]{Acceleration}, which is the continuous-time limit of the algorithm. Using the Lyapunov function
\begin{align*}
\mathcal{E}_t = t^p(f(X_t) - f(x)),
\end{align*}
the following argument can be made using the convexity of $f$ and the dynamics:
\begin{align*}
\dot{\mathcal{E}}_t &=  t^p\langle \nabla f(X_t), \dot X_t\rangle  + p t^{p-1}(f(X_t) - f(x^\ast))\\
&\leq t^p\langle \nabla f(X_t), \dot X_t\rangle +  p t^{p-1}\langle \nabla f(X_t), X_t - x^\ast\rangle\\
&=  - t^p\|\nabla f(X_t)\|_\ast^{\frac{p}{p-1}} + p t^{p-1}\langle \nabla f(X_t), X_t - x^\ast \rangle\\
&\leq \frac{1}{p-1}\|(p-1)(X_t - x^\ast)\|^p\\
& \leq (p-1)^{p-1} R^p.
\end{align*}
The last two inequalities use the Fenchel-Young inequality and the fact that 
$\|X_t - x^\ast\|\leq R$ since rescaled gradient flow is a descent method. 
We can conclude $O(t^{p-1})$ convergence rate by integrating. We now proceed with the discrete-time argument by using the Lyapunov function~\eqref{eq:lyapfrank} ($\tilde p \geq 2$):
\begin{align*}
E_k = A_k(f(x_k) - f(x))
\end{align*}
We argue as follows:
\begin{align*}
E_{k+1} - E_k &= A_{k} (f(x_{k+1}) - f(x_k)) + \alpha_k (f(x_{k+1}) - f(x^\ast))\\
&\leq A_{k} \langle \nabla f(x_{k+1}), x_{k+1} - x_k \rangle + \alpha_k \langle \nabla f(x_{k+1}), x_{k+1} - x^\ast\rangle\\
&\leq -A_{k} \frac{(N^2-1)^{\frac{\tilde p-2}{2 \tilde p-2}}}{2N} \epsilon^{\frac{1}{\tilde p-1}} \|\nabla f(x_{k+1})\|_*^{\frac{\tilde p}{\tilde p-1}} + \alpha_k \langle \nabla f(x_{k+1}),x_{k+1} - x^\ast\rangle\\
&\leq  \frac{1}{\epsilon}\alpha_k^{\tilde p} A_k^{1-\tilde p} \frac{(\tilde p-1)^{\tilde p-1}}{\tilde p^{\tilde p}}  \left(\frac{(N^2-1)^{\frac{\tilde p-2}{2 \tilde p-2}}}{2N}\right)^{-\frac{\tilde p-1}{\tilde p}} \|x_{k+1} - x^\ast\|^{\tilde p},
\end{align*}
where the first inequality follows from convexity, the second inequality 
uses~\ref{prop:Hold}, and the third line uses Young's inequality, $\langle s,u\rangle + \frac{1}{p}\|u\|^p \leq -\frac{p-1}{p}\|s\|_\ast^{\frac{p-1}{p}}$, $2\geq p\in \mathbb{R}$  with the identifications 
\begin{align*} s &= \epsilon^{1/\tilde{p}} \nabla f(x_{k+1}) \left(A_k\frac{(N^2-1)^{\frac{\tilde p-2}{2 \tilde p-2}}}{2N}\right)^{\frac{\tilde p-1}{\tilde p}} \\
u &= (x_{k+1} - x^\ast) \epsilon^{-\frac{1}{\tilde p}} \left(A_k\frac{(N^2-1)^{\frac{\tilde p-2}{2 \tilde p-2}}}{2N}\right)^{-\frac{\tilde p-1}{\tilde p}} \alpha_k \frac{\tilde p-1}{\tilde p}.
\end{align*}
From Lemma~\ref{prop:Hold}, it follows that this method is a descent method. 
Furthermore, we can choose $\alpha_k^{\tilde {p}} A_k^{1-\tilde p}\leq C$, for some constant $C$,  by choosing $A_k$ to be a polynomial of degree $\tilde p$. By summing we 
obtain the desired $O(1/k^{\tilde p - 1})$ convergence rate.

\subsection{Proof of Proposition~\ref{prop:strongbound}}
\label{App:strongbound_proof}
We show the initial error bound~\eqref{eq:errrrr}. To do so, we define the Lyapunov function,
\begin{align}\label{eq:lyapplain}
\tilde E_k = f(y_k) - f(x^\ast) + \mu D_h(x, z_k).
\end{align}
Note that we simply need to show $\tilde E_{k+1} - \tilde E_k \leq -\tau_k \tilde E_k + \varepsilon_{k+1}/A_{k+1}$ where $\tau_k = \frac{A_{k+1} - A_k}{A_{k+1}}$. Thus, we begin with the following bound:
\begin{align*}
\tilde E_{k+1} -\tilde  E_k &= f(y_{k+1}) - f(y_k) - \mu \langle \nabla h(z_{k+1}) - \nabla h(z_k), x^\ast - z_{k+1}\rangle - \mu D_h(z_{k+1} ,z_k)\\
&\leq f(y_{k+1}) - f(x_k) + f(x_k) - f(y_k) - \mu \langle \nabla h(z_{k+1}) - \nabla h(z_k), x^\ast - z_{k}\rangle+ \frac{\sigma \mu}{2}\|\nabla h(z_{k+1}) - \nabla h(z_k)\|^2\\
&\leq f(y_{k+1}) - f(x_k) + \langle \nabla f(x_k), x_k - y_k\rangle - \mu D_h(x_k,y_k) - \mu \langle \nabla h(z_{k+1}) - \nabla h(z_k), x - z_{k}\rangle\\
& + \frac{\sigma \mu}{2}\|\nabla h(z_{k+1}) - \nabla h(z_k)\|^2 \\
&\overset{\eqref{Eq:ZSeq1}}{=} f(y_{k+1}) - f(x_k) + \langle \nabla f(x_k), x_k - y_k\rangle -  \mu D_h(x_k,y_k)   + \tau_k \langle \nabla f(x_{k}), x - z_k\rangle \\
&\quad   - \mu\tau_k\langle \nabla h(x_{k}) - \nabla h(z_k), x^\ast - z_{k}\rangle+ \frac{\sigma \mu}{2}\|\nabla h(z_{k+1}) - \nabla h(z_k)\|^2\\
&\overset{\eqref{Eq:Coupling1}}{=} f(y_{k+1})- f(x_k) + \langle \nabla f(x_k), x_k - y_k\rangle - \mu D_h(x_k,y_k)  + \tau_k \langle \nabla f(x_{k}), x - x_{k} \rangle\\
&\quad  - \tau_k\mu\langle \nabla h(x_{k}) - \nabla h(z_k), x - z_{k}\rangle + \langle\nabla f(x_{k}), y_{k} - x_{k}\rangle + \frac{\sigma \mu}{2}\|\nabla h(z_{k+1}) - \nabla h(z_k)\|^2\\
%
&\leq -\tau_k\left( f(x_{k} ) - f(x^\ast) + \mu D_h(x, x_{k}) \right) + f(y_{k+1}) - f(x_k)  - \frac{\sigma \mu}{2}\|x_k - y_k\|^2   \\
&\quad - \tau_k\mu\langle \nabla h(x_{k}) -\nabla h(z_k), x^\ast - z_{k}\rangle  + \frac{\sigma \mu}{2}\|\nabla h(z_{k+1}) - \nabla h(z_k)\|^2\\
&= -\tau_k\left( f(y_{k} ) - f(x^\ast) + \mu D_h(x,z_k) \right) + f(y_{k+1}) - f(x_k)  - \frac{\mu\sigma}{2}\|x_k - y_k\|^2  \\
&\quad   +  \frac{\sigma \mu}{2}\|\nabla h(z_{k+1}) - \nabla h(z_k)\|^2 -\tau_k\mu D_h(x_k,z_k) + \tau_k(f(y_k)) - f(x_k))\\
&\leq -\tau_k \tilde E_k  + f(y_{k+1}) - f(x_k)  - \frac{\mu\sigma}{2}\|x_k - y_k\|^2+ \tau_k\langle \nabla f(x_k), y_k - x_k\rangle  \\
&\quad   +  \frac{\sigma \mu}{2}\|\tau_k(\nabla h(x_k) - \nabla h(z_k) - \frac{1}{\mu}\nabla f(x_k))\|^2 -\left( \frac{\sigma\mu}{2\tau_k} - \frac{\tau_k}{2\epsilon}\right)\|x_k - y_k\|^2\\
&\leq -\tau_k \tilde E_k +  \varepsilon_{k+1} /A_{k+1}.
\end{align*}
The first inequality uses the $\sigma$-strong convexity of $h$ and the 
Fenchel-Young inequality. The second inequality uses the $\mu$-strong 
convexity of $f$ with respect to $h$.  The third inequality uses the 
strong convexity of $f$ and $\sigma$-strong convexity of $h$. The following line uses the Bregman three point identity~\eqref{eq:bregthree} and the subsequent 
inequality uses the strong convexity of $f$. The last line follows 
from the smoothness of $f$. Now we turn to the case where $h$ is 
Euclidean (so $\sigma = 1$):
\begin{align*}
\tilde E_{k+1} - \tilde E_k &\leq -\tau_k \tilde E_k  + f(y_{k+1}) - f(x_k)  - \frac{\mu}{2}\|x_k - y_k\|^2+ \tau_k\langle \nabla f(x_k), y_k - x_k\rangle  \\
&\quad   +  \frac{ \mu}{2}\|\tau_k(x_k - z_k) - \frac{\tau_k}{\mu}\nabla f(x_k))\|^2 -\left( \frac{\mu}{2\tau_k} - \frac{\tau_k}{2\epsilon}\right)\|x_k - y_k\|^2\\
&= -\tau_k \tilde E_k  + f(y_{k+1}) - f(x_k)  - \frac{\mu}{2}\|x_k - y_k\|^2+ \tau_k\langle \nabla f(x_k), y_k - x_k\rangle   +  \frac{ \mu}{2}\|\tau_k(x_k - z_k)\| \\
&\quad  - \tau_k \langle \nabla f(x_{k}), \tau_k(x_{k} - z_k)\rangle + \frac{\tau_k^2}{2\mu}\|\nabla f(x_{k})\|^2  -\left( \frac{\mu}{2\tau_k} - \frac{\tau_k}{2\epsilon}\right)\|x_k - y_k\|^2\\
&=-\tau_k \tilde E_k  + f(y_{k+1}) - f(x_k)   + \frac{\tau_k^2}{2\mu}\|\nabla f(x_{k})\|^2  -\left( \frac{\mu}{2\tau_k} - \frac{\tau_k}{2\epsilon}\right)\|x_k - y_k\|^2.
\end{align*}
In the second line we have expanded the square. The last line uses the update~\eqref{Eq:Coupling1}.

\subsection{Proof of Proposition~\ref{prop:quasi-strong}}
\label{App:quasi-strong}
We show the convergence bound for the quasi-monotone method~\eqref{eq:str1}. We have,
\begin{align*}
\tilde E_{k+1} - \tilde E_k &= - \langle \nabla h(z_{k+1}) - \nabla h(z_k), x - z_{k+1}\rangle - D_h(z_{k+1}, z_k) + f(x_{k+1}) - f(x_k)\\
&\overset{\eqref{eq:strz}}{=} \tau_k \langle \nabla f(x_{k+1}), x - z_{k} \rangle +\tau_k \langle \nabla h(z_{k+1}) - \nabla h(x_{k+1}), x- z_{k+1}\rangle \\&\quad +\tau_k \langle \nabla f(x_{k+1}), z_k- z_{k+1} \rangle - D_h(z_{k+1}, z_k) + f(x_{k+1}) - f(x_k)\\
&\leq \tau_k \langle \nabla f(x_{k+1}), x - z_{k} \rangle+ \tau_k \langle \nabla h(z_{k+1}) - \nabla h(x_{k+1}), x- z_{k+1}\rangle \\
&\quad +\frac{\tau_k^2\sigma }{2} \|\nabla f(x_{k+1})\|^2 + f(x_{k+1}) - f(x_k)\\
&= \tau_k \langle \nabla f(x_{k+1}), x - x_{k+1} \rangle+ \tau_k \langle \nabla h(z_{k+1}) - \nabla h(x_{k+1}), x- z_{k+1}\rangle \\
&\quad +\frac{\tau_k^2\sigma }{2} \|\nabla f(x_{k+1})\|^2 + f(x_{k+1}) - f(x_k)  + \langle \nabla f(x_{k+1}), x_{k} -x_{k+1}\rangle\\
&\leq -\tau_k( f(x_{k+1}) - f(x^\ast) + D_h(x,x_{k+1}))   + \tau_k \langle \nabla h(z_{k+1}) - \nabla h(x_{k+1}), x- z_{k+1}\rangle\\
&\quad +\frac{\tau_k^2\sigma }{2} \|\nabla f(x_{k+1})\|^2 \\
&\leq -\tau_k( f(x_{k+1}) - f(x^\ast) + D_h(x,z_{k+1})) +\frac{\tau_k^2\sigma }{2} \|\nabla f(x_{k+1})\|^2 
\end{align*}
The first inequality from the strong convexity of $h$ as well as H\"older's 
inequality. The second inequality from the uniform convexity of $f$ with respect to $h$ and convexity of $f$. The last line follows from the Bregman three-point identity~\eqref{eq:bregthree} and non-negativity of the Bregman divergence. Taking $\tau_k = \frac{A_{k+1} - A_k}{A_k}$ gives the desired error bound. 

\subsection{Proof of Proposition~\ref{prop:franklyap}}
\label{App:franklyap}
We show that~\eqref{eq:lyapfrank} is a Lyapunov function for dynamics \eqref{Eq:Frankd}. The argument is simple:
\begin{align*}
0 \leq \dot\beta_te^{\beta_t}\langle \nabla f(X_t), x -Z_t \rangle &= \dot\beta_te^{\beta_t}\langle \nabla f(X_t), x -X_t \rangle - e^{\beta_t}\langle \nabla f(X_t), \dot X_t \rangle \\
&=  \dot\beta_te^{\beta_t}\langle \nabla f(X_t), x -X_t \rangle - \frac{d}{dt}\left\{e^{\beta_t}f(X_t)\right\} + \dot \beta_t e^{\beta_t} f(X_t) \\
&\leq  -\frac{d}{dt}\left\{e^{\beta_t}(f(X_t) - f(x))\right\}  .
\end{align*}
\subsection{Proof of Proposition~\ref{eq:propFrank}}
\label{App:prop_Frank} 
If we take $\nu = 1$ bound~\eqref{eq:bound2} implies~\eqref{eq:bound1}; therefore we simply show the bound~\eqref{eq:bound2}. To that end,
\begin{align*}
E_{k+1} - E_k 
&= A_{k+1}(f(x_{k+1}) - f(x_k)) + \alpha_k (f(x_k) - f(x))\\
&\leq A_{k+1}\langle \nabla f(x_k), x_{k+1} - x_k\rangle + \frac{A_{k+1}}{(1+\nu)\epsilon}\|x_{k+1} - x_k\|^{1+\nu} + \alpha_k \langle \nabla f(x_k), x - x_k\rangle \\
&\overset{\eqref{Eq:ZSeqFrank1}}{=} \alpha_k\langle \nabla f(x_k), z_k - x_k\rangle + \frac{A_{k+1}\alpha_k^{1+\nu}}{(1+\nu)\epsilon}\|z_k - x_k\|^{1+\nu}+ \alpha_k \langle \nabla f(x_k), x - x_k\rangle \\
&\overset{\eqref{Eq:XSeqFrank1}}{\leq}\frac{A_{k+1}\alpha_k^{1+\nu}}{(1+\nu)\epsilon}\|z_k - x_k\|^{1+\nu}.
\end{align*}
The first inequality follows from the H\"older continuity and convexity of $f$. The rest simply follows from plugging in our identities.
\section{Estimate Sequences}
\label{App:EstSeq}
\subsection{The quasi-monotone subgradient method}
The discrete-time estimate sequence~\eqref{Eq:Est} for quasi-monotone subgradient method can be written:
\begin{align*}
\phi_{k+1}(x) - A_{k+1}^{-1}\tilde \varepsilon_{k+1}&:= f(x_{k+1}) + A_{k+1}^{-1} D_h(x, z_{k+1}) - A_{k+1}^{-1}\tilde \varepsilon_{k+1}\\
& \overset{\eqref{Eq:Est}}{=} (1 - \tau_k) \left(\phi_k(x) - A_k^{-1}\tilde \varepsilon_k \right) + \tau_k f_k(x) \\
&= \left(1 - \frac{\alpha_k}{A_{k+1}} \right)\left(f(x_k) +\frac{1}{A_k} D_h(x, z_k) -\frac{\tilde \varepsilon_k}{A_k}\right) + \frac{\alpha_k}{A_{k+1}} f_k(x).
\end{align*}
Multiplying through by $A_{k+1}$, we have
\begin{align*}
A_{k+1} f(x_{k+1}) + D_h(x, z_{k+1}) - \tilde \varepsilon_{k+1} &= (A_{k+1} - \alpha_k) (f(x_k) + A_k^{-1}D_h(x, z_k) -A_k^{-1}\tilde{\varepsilon}_k) \\
&\quad  - (A_{k+1} - \alpha_k) A_k^{-1}\tilde{\varepsilon}_k + \alpha_k f_k(x)\\
&=  A_k \left(f(x_k) + A_k^{-1}D_h(x, z_k) - A_k^{-1}\tilde \varepsilon_k\right) + \alpha_k f_k(x)\\
& \overset{\eqref{Eq:Under}}{\leq}  A_k f(x_k) + D_h(x, z_k) - \tilde \varepsilon_k + \alpha_k f(x).
\end{align*}
Rearranging, we obtain our Lyapunov argument $E_{k+1} \leq E_k + \varepsilon_{k+1}$ for~\eqref{eq:lyap}:
\begin{align*}
A_{k+1}(f(x_{k+1}) - f(x)) + D_h(x, z_{k+1}) &\leq A_k (f(x_k) - f(x)) + D_h(x, z_k) + \varepsilon_{k+1}.
\end{align*}
%
Going the other direction, from our Lyapunov analysis we can derive the following bound:
\begin{align}
E_k &\leq E_0 + \tilde \varepsilon_{k}\\
A_{k}(f(x_{k}) - f(x)) + D_h(x, z_{k}) &\leq A_0 (f(x_0) - f(x)) + D_h(x, z_0) + \tilde \varepsilon_{k}\notag\\
A_{k}\left(f(x_{k}) - \frac{1}{A_{k}}D_h(x, z_{k})\right) &\leq (A_k -  A_0) f(x) + A_0\left(f(x_0)+ \frac{1}{A_0}  D_h(x^\ast, z_0)\right) + \tilde \varepsilon_{k}\notag \\
A_k \phi_k(x) &\leq (A_k -A_0) f(x) + A_0 \phi_0(x) + \tilde \varepsilon_{k}.
\end{align}
Rearranging, we obtain our estimate sequence~\eqref{Eq:Ineq1} ($A_0 = 1$) with an additional error term:
\begin{subequations}
\begin{align}
\phi_k(x) &\leq \Big(1 -\frac{A_0}{A_k}\Big) f(x) + \frac{A_0}{A_k} \phi_0(x) + \frac{\tilde \varepsilon_{k}}{A_k}= \Big(1 - \frac{1}{A_k}\Big)f(x) + \frac{1}{A_k}\phi_0(x)+ \frac{\tilde \varepsilon_{k}}{A_k}.
\end{align}
\end{subequations}
%
\subsection{Frank-Wolfe}
The discrete-time estimate sequence~\eqref{Eq:Est} for conditional gradient method can be written:
\begin{align*}
\phi_{k+1}(x) - \frac{\tilde \varepsilon_{k+1}}{A_{k+1}} := f(x_{k+1})  - \frac{\tilde \varepsilon_{k+1}}{A_{k+1}}&\overset{\eqref{Eq:Est}}{=} (1 - \tau_k) \left(\phi_k(x) - \frac{\tilde \varepsilon_k}{A_k}\right) + \tau_k f_k(x) \\
&\overset{\text{Table}~\ref{table:Table}}{=} \left(1 - \frac{\alpha_k}{A_{k+1}} \right)\left(f(x_k) -\frac{\tilde \varepsilon_k}{A_k}\right)+ \frac{\alpha_k}{A_{k+1}} f_k(x).
\end{align*}
Multiplying through by $A_{k+1}$, we have
\begin{align*}
A_{k+1} \left(f(x_{k+1}) - \frac{\tilde \varepsilon_{k+1}}{A_{k+1}}\right) &= (A_{k+1} - (A_{k+1} - A_k)) \left(f(x_k) - \frac{\tilde \varepsilon_k}{A_k}\right) + \alpha_k f_k(x)\\
&=  A_k \left(f(x_k) - A_k^{-1}\tilde \varepsilon_k\right) + (A_{k+1} - A_k) f_k(x)\\
& \overset{\eqref{Eq:Under}}{\leq}  A_k f(x_k)  - \tilde \varepsilon_k + (A_{k+1} - A_k) f(x).\
\end{align*}
Rearranging, we obtain our Lyapunov argument $E_{k+1} - E_k \leq \varepsilon_{k+1}$ for~\eqref{eq:lyapfrank} :
\begin{align*}
A_{k+1}(f(x_{k+1}) - f(x))  &\leq A_k (f(x_k) - f(x)) + \varepsilon_{k+1}.
\end{align*}
Going the other direction, from our Lyapunov analysis we can derive the following bound:
\begin{subequations}
\begin{align*}
E_k &\leq E_0 + \tilde \varepsilon_{k}\\
A_{k}f(x_{k}) &\leq (A_k -  A_0) f(x) + A_0f(x_0) + \tilde \varepsilon_{k}\notag \\
A_k \phi_k(x) &\leq (A_k -A_0) f(x) + A_0 \phi_0(x) + \tilde \varepsilon_{k}
\end{align*}
\end{subequations}
Rearranging, we obtain our estimate sequence~\eqref{Eq:Ineq1} ($A_0 = 1$) with an additional error term:
\begin{subequations}
\begin{align*}
\phi_k(x) &\leq \Big(1 -\frac{A_0}{A_k}\Big) f(x) + \frac{A_0}{A_k} \phi_0(x) + \frac{\tilde \varepsilon_{k}}{A_k}=\Big(1 - \frac{1}{A_k}\Big)f(x) + \frac{1}{A_k}\phi_0(x)+ \frac{\tilde \varepsilon_{k}}{A_k}.
\end{align*}
\end{subequations}
Given that the Lyapunov function property allows us to write 
\begin{align*}
e^{\beta_t} f(X_t) \leq  (e^{\beta_t} - e^{\beta_0}) f(x) + e^{\beta_0}f(X_0),
\end{align*}
we can extract $\{f(X_t),e^{\beta_t}\}$ as the continuous-time estimate sequence for Frank-Wolfe. %
\subsection{Accelerated gradient descent (strong convexity)}
The discrete-time estimate sequence~\eqref{Eq:Est} for accelerated gradient descent can be written:
\begin{align*}
\phi_{k+1}(x) &:= f(x_{k+1}) + \frac{\mu}{2}\|x -  z_{k+1}\|^2 \overset{\eqref{Eq:Est}}{=} (1 - \tau_k) \phi_k(x)  + \tau_k f_k(x)\overset{\eqref{Eq:Under}}{\leq}(1 - \tau_k) \phi_k(x)  + \tau_k f(x).
\end{align*}
Therefore, we obtain the inequality $\tilde E_{k+1} - \tilde E_k \leq -\tau_k \tilde E_k$ for our Lyapunov function~\eqref{eq:lyapplain} by simply writing $\phi_{k+1}(x) - f(x) + f(x) - \phi_k(x) \leq - \tau_k(\phi_k(x)  -  f(x))$:
\begin{multline*}
f(x_{k+1}) -f(x) + \frac{\mu}{2}\|x - z_{k+1}\|^2 - \left(f(x_k) - f(x)  + \frac{\mu}{2}\|x - z_{k+1}\|^2\right) \\\overset{\text{Table}~\ref{table:Table}}{\leq}- \tau_k \left(f(x_k) - f(x)  + \frac{\mu}{2}\|x - z_{k+1}\|^2\right). 
\end{multline*}
Going the other direction, we have, 
\begin{align*}
E_{k+1} - E_k &\leq -\tau_k E_k\\
\phi_{k+1} &\leq (1 - \tau_k) \phi_k(x)  + \tau_k f(x)\\
A_{k+1} \phi_{k+1} &\leq A_k \phi_k + (A_{k+1} - A_k) f(x).
\end{align*}
Summing over the right-hand side, we obtain the estimate sequence~\eqref{Eq:Ineq1}: 
\begin{align*}
\phi_{k+1} &\leq \Big(1 -\frac{A_0}{A_{k+1}}\Big) f(x) + \frac{A_0}{A_{k+1}} \phi_0(x)=  \Big(1 -\frac{1}{A_{k+1}}\Big) f(x) + \frac{1}{A_{k+1}} \phi_0(x).
\end{align*}
Since the Lyapunov function property allows us to write 
 \begin{align*}
e^{\beta_t} \left(f(X_t) +\frac{\mu}{2} \|x - Z_t\|^2\right) \leq (e^{\beta_t} - e^{\beta_0}) f(x) + e^{\beta_0}  \left(f(X_0) + \frac{\mu}{2} \|x - Z_0\|^2\right),
\end{align*}
we can extract $\{f(X_t) + \frac{\mu}{2} \|x - Z_t\|^2, e^{\beta_t}\}$ as the continuous-time estimate sequence for accelerated gradient descent in the strongly convex setting.
\subsection{Existence and uniqueness}
\label{Sec:Exist}
In this section, we show existence and uniqueness of solutions for the differential equations~\eqref{eq:el2}, when $h$ is Euclidean. To do so, we write the dynamics as the following system of equations
\begin{subequations}\label{eq:exist}
\begin{align}
\dot X_t &= \sqrt{\mu}(W_t - 2X_t)\\
\dot W_t &= -\frac{1}{\sqrt{\mu}} \nabla f(X_t),
\end{align}
\end{subequations}
where we have taken $W_t = Z_t + X_t$ and $\beta_t = \sqrt{\mu} t$. Now if we assume $\nabla f$ is Lipschitz continuous, then over any bounded interval $[t_0, t_1]$ with $0\leq t_0< t_1$, the right-hand side of~\eqref{eq:exist} is a Lipschitz-continuous vector field. Therefore, by the Cauchy-Lipschitz theorem, for any initial conditions $(X_{t_0}, W_{t_0}) = (x_0, w_0)$ at time $t = t_0$, the system of differential equations has a unique solution over the time interval $[t_0, t_1]$. Since $t_1$ is arbitrary and the energy is decreasing in $t_1$, this shows that there is a unique maximal solution for any $t_1 \rightarrow \infty$. To show a unique solution exists for an arbitrary $\beta_t$, we show the family of dynamics~\eqref{eq:el2} is closed under time-dilation (similar to dynamics~\eqref{eq:el1}). Thus, if a unique solution exists for any setting $\beta_t$, we can conclude it exists for all $\beta_t$. To demonstrate the time-dilation property, we calculate the velocity and acceleration of the reparameterized curve $Y_t = X_{\tau_t}$, where $\tau: \mathbb{R}_+ \rightarrow \mathbb{R}_+$ is an increasing function of time:
\begin{align*}
\dot Y_t &= \dot \tau_t \dot X_{\tau_t}\\
\ddot Y_t &= \ddot \tau_t \dot X_{\tau_t}+ \dot \tau_t^2 \ddot X_{\tau_t}\\
\dot{\tilde\beta}_t &= \frac{d}{dt} \beta_{\tau_t} = \dot \tau_t\dot{\beta}_{\tau_t} \\
\ddot {\tilde\beta}_t &= \frac{d}{dt} \dot \tau_t \dot \beta_{\tau_t} = \ddot \tau_t \dot   \beta_{\tau_t} + \dot \tau_t^2\ddot  \beta_{\tau_t}.
\end{align*}
Inverting the first of these relations, we get
\begin{align*}
\dot X_{\tau_t} &= \frac{1}{\dot \tau_t} \dot Y_t\\
\ddot X_{\tau_t} &= \frac{1}{\dot \tau_t^2} \ddot Y_t - \frac{\ddot \tau_t}{\dot \tau_t^3} \dot Y_t. \\
\dot{\beta}_{\tau_t}&= \frac{1}{\dot \tau_t}\dot{\tilde\beta}_t  \\
\ddot  \beta_{\tau_t} &= \frac{1}{\dot \tau_t^2}\ddot {\tilde\beta}_t - \frac{\ddot \tau_t }{ \dot \tau_t^2}\dot   \beta_{\tau_t}.
\end{align*}
Computing the time-dilated Euler-Lagrange equation, we get
\begin{align*}
Z_{\tau_t} &= X_{\tau_t} + \frac{1}{\dot \beta_{\tau_t}} \dot X_{\tau_t}= Y_t +  \frac{1}{\dot \beta_{\tau_t}\dot \tau_t} \dot Y_{t} = Y_t +  \frac{1}{\dot{\tilde \beta}_{t}} \dot Y_{t}
\end{align*}
for the first equation, as well as the identity
\begin{align*}
\dot Z_{\tau_t} &= \dot X_{\tau_t}+ \frac{1}{\dot \beta_{\tau_t}}\ddot X_{\tau} - \frac{\ddot \beta_{\tau_t}}{\dot \beta_{\tau_t}^2} \dot X_{\tau_t} \\
& = \frac{1}{\dot \tau_t} \dot Y_t + \frac{1}{\dot \tau_t\dot{\tilde \beta}_{\tau_t}}\ddot Y_t - \frac{\ddot \tau_t}{\dot{\tilde{\beta}}_t\dot \tau_t^2} \dot Y_t - \frac{\ddot \beta_{\tau_t}}{\dot \tau_t\dot \beta_{\tau_t}^2 } \dot Y_t\\
& = \frac{1}{\dot \tau_t} \dot Y_t + \frac{1}{\dot \tau_t\dot{\tilde \beta}_{\tau_t}}\ddot Y_t - \frac{\ddot \tau_t}{\dot{\tilde{\beta}}_t\dot \tau_t^2} \dot Y_t - \frac{\ddot {\tilde\beta}_t }{\dot \tau_t\dot{\tilde{\beta}}_{t}^2 } \dot Y_t + \frac{\ddot \tau_t}{\dot \tau_t^2 \dot{\tilde{ \beta}}_{\tau_t} } \dot Y_t\\
&= \frac{1}{\dot \tau_t}\left( \dot Y_t  + \frac{1}{\dot{\tilde \beta}_{\tau_t}}\ddot Y_t -  \frac{\ddot {\tilde\beta}_t }{\dot{\tilde{\beta}}_{t}^2 } \dot Y_t \right).
\end{align*}
Therefore the second equation
\begin{align*}
\nabla^2 h(Z_{\tau_t})\dot Z_{\tau_t} = - \dot \beta_{\tau_t}(\nabla h(X_{\tau_t}) -  \nabla h(Z_{\tau_t}) -  \frac{1}{\mu}\nabla f(X_{\tau_t}))
\end{align*}
can be written, 
\begin{align*}
\frac{1}{\dot \tau_t}\left(\nabla^2 h\left(Y_t +  \frac{1}{\dot{\tilde \beta}_{t}} \dot Y_{t}\right) \left( \dot Y_t  + \frac{1}{\dot{\tilde \beta}_{\tau_t}}\ddot Y_t -  \frac{\ddot {\tilde\beta}_t }{\dot{\tilde{\beta}}_{t}^2 } \dot Y_t \right) \right)= - \frac{\dot{\tilde{\beta}}_{t}}{\dot \tau_t}\left(\nabla h(Y_t)- \nabla h\left(Y_t +  \frac{1}{\dot{\tilde \beta}_{t}} \dot Y_{t}\right) -  \frac{1}{\mu}\nabla f(Y_t)\right),
\end{align*}
which is the Euler-Lagrange equation for the sped-up curve, where the ideal 
scaling holds with equality. 
Finally, we mention that we can deduce the existence/uniqueness of solution for the proximal dynamics~\eqref{eq:prox_dyn} and \eqref{eq:prox_dyn2} from the  existence/uniqueness of solution for dynamics~\eqref{eq:el1} and~\eqref{eq:el2}, given the difference between these dynamics is that~\eqref{eq:prox_dyn}~\eqref{eq:prox_dyn2} have an extra Lipschitz-continuous vector field. Thus, the Cauchy-Lipschitz theorem can be readily applied to the proximal dynamics and the same arguments can be made regarding time-dilation. 

\section{Additional Observations}
\label{Sec:AddObv}
  \subsection{Proximal algorithms}
  \label{Sec:Prox}
  \subsubsection{Convex functions~\cite{Tseng08,BeckTeboulle09,Nesterov13}}
In 2009, Beck and Teboulle introduced FISTA, which is a method for minimizing the composite of two convex functions 
\begin{align}\label{eq:composite}
f(x) = \varphi(x) + \psi(x)
\end{align}
where $\varphi$ is  $(1/\epsilon)$-smooth and $\psi$ is simple. 
The canonical example of this is $\psi(x) = \|x\|_1$, which
defines the $\ell_1$-ball. The following proposition provides 
dynamical intuition for momentum algorithms derived for this setting.
\begin{proposition}\label{prop:prox_dynamics}
Define $f = \varphi + \psi$ and assume $\varphi$ and $\psi$ are convex. Under the ideal scaling condition~\eqref{Eq:IdeScaBet}, Lyapunov function~\eqref{eq:Lyap1} can be used to show that solutions to dynamics
\begin{subequations}\label{eq:prox_dyn}
\begin{align}
Z_t &= X_t + e^{-\alpha_t} \dot X_t \label{eq:zup}\\
\frac{d}{dt}\nabla h(Z_t) &= -e^{\alpha_t+ \beta_t} (\nabla \varphi(X_t)+ \nabla \psi(Z_t))\label{eq:comp}
\end{align}
\end{subequations}
satisfy  $f(X_t) - f(x^\ast) \leq O(e^{-\beta_t})$.
\end{proposition}
The same Lyapunov argument can be made for the dynamics \eqref{eq:prox_dyn} if we replace $\nabla \psi(Z_t)$ with a directional subgradient at the position $Z_t$, provided $\beta_t = p\log t$ for $p \in \mathbb{R}$. 
\begin{proof}
\begin{align*}
 \frac{d}{dt}D_h\left(x, Z_t\right)  
 &= -\left\langle \frac{d}{dt} \nabla h\left(Z_t\right), x - Z_t\right\rangle \, \notag\\
&=    e^{\alpha_t +\beta_t}\left\langle\nabla \varphi(X_t), x  - X_{t} - e^{-\alpha_t} \dot X_{t} \right\rangle  + e^{\alpha_t +\beta_t}\left\langle\nabla \psi(Z_t), x  -Z_t \right\rangle\, \\ 
&\leq    - \frac{d}{dt} \Big\{e^{\beta_t}\left(\varphi(X_t) - \varphi(x)\right)\Big\}  + e^{\alpha_t +\beta_t}\left\langle\nabla \psi(Z_t), x  -Z_t \right\rangle\,\\ 
&\leq  - \frac{d}{dt} \Big\{e^{\beta_t}\left(\varphi(X_t) - \varphi(x)\right)\Big\}  + \dot \beta_t e^{\beta_t}( \psi(x) - \psi(Z_t))\, \\
&\leq  - \frac{d}{dt} \Big\{e^{\beta_t}\left(\varphi(X_t) - f(x)\right)\Big\}- \dot \beta_t e^{\beta_t}(\psi(X_t) +\langle \nabla \psi(X_t), Z_t - X_t\rangle) \,\\
&= - \frac{d}{dt} \Big\{e^{\beta_t}\left(\varphi(X_t) - f(x)\right)\Big\}- \dot \beta_t e^{\beta_t}\psi(X_t) -e^{\beta_t}\langle \nabla \psi(X_t), \dot X_t\rangle) \, \\
&= - \frac{d}{dt} \Big\{e^{\beta_t}\left(f(X_t) - f(x)\right)\Big\}.
\end{align*}
The first line follows from the Bregman identity~\eqref{eq:iden}. 
The second line plugs in the dynamics \eqref{eq:zup2} and \eqref{eq:comp2}. 
The third lines follows from~\eqref{Eq:LyapAnal}. The fourth and fifth 
lines follow from convexity. The sixth line plugs in the dynamics~\eqref{eq:comp2} 
and the last line follows from application of the chain rule.
\end{proof}

Next, to show results for dynamics when subgradients of the function are used, 
we adopt the setting of Su, Boyd and Candes~\cite[p.35]{SuBoydCandes}. 
First, we define the subgradient through the following lemma.
\begin{lemma}[Rockafellar, 1997] For any convex function $f$ and any $x, v \in \mathbb{R}^n$, the directional derivative $\lim_{\delta \rightarrow 0+}(f(x+\delta v) - f(x))/\delta$ exists, and can be evaluated as
\begin{align*}
\lim_{\delta \rightarrow 0+}(f(x+\delta v) - f(x))/\delta = \sup_{w \in \partial f(x)}\langle w, v\rangle.
\end{align*}
\end{lemma}
\begin{definition}
A Borel-measurable function $G_f(x,v)$ defined on $\mathbb{R}^n \times \mathbb{R}^n$ 
is said to be a directional subgradient of $f$ if 
\begin{align*}
G_f(x,v) &\in \partial f(X)\\
\langle G_f(x,v),v\rangle& = \sup_{w\in \partial f(x)}\langle w, v\rangle,
\end{align*}
for all $x,v$.
\end{definition} 
This guarantees the existence of a directional derivative.  Now we establish the following theorem (similar to \cite[Thm 24]{SuBoydCandes}):
\begin{theorem}
Given the sum of two convex functions $f(x) = \varphi(x) + \psi(x)$ with directional subgradient $G_\psi(x,v)$, assume that the second-order ODE 
\begin{align*}
Z_t &= X_t + \frac{t}{p} \dot X_t\\
\frac{d}{dt} \nabla h(Z_t) &= -pt^{p-1}(G_\varphi(X_t, \dot X_t) + G_\psi(Z_t, \dot Z_t))
\end{align*}
admits a solution $X_t$ on $[0,\alpha)$ for some $\alpha>0$. Then for any $0<t<\alpha$, we have $f(X_t) - f(x) \leq O(1/t^p)$.
\end{theorem}
\begin{proof}
We follow the framework of Su, Boyd and Candes~\cite[pg. 36]{SuBoydCandes}. It suffices to establish that our Lyapunov function is monotonically decreasing. Although $\mathcal{E}_t$ may not be differentiable, we can study $\mathcal{E}(t + \Delta t) - \mathcal{E}(t))/\Delta t$ for small $\Delta t>0$.
For the first term, note that
\begin{align*}
(t+ \Delta t)^p( f(X_{t + \Delta t}) - f(x)) - t^p(f(X_t) - f(x)) &= t^p (f(X_{t + \Delta t}) - f(X_t))\\
&  + pt^{p-1}(f(X_{t+\Delta t}) - f(x))\Delta t + o(\Delta t)\\
&= t^p \langle G_f(X_t, \dot X_t), \dot X_t\rangle \Delta t  \\
&  + pt^{p-1}(f(X_{t+\Delta t}) - f(x))\Delta t + o(\Delta t),
\end{align*}
where the second line follows since we assume $f$ is locally Lipschitz. 
The $o(\Delta t)$ does not affect the function in the limit: 
\begin{align}\label{eq:approx}
f(X_{t + \Delta t}) = f(X + \Delta t \dot X_t + o(\Delta t)) &= f(X + \Delta t \dot X_t) + o(\Delta t)\notag\\
&= f(X_t) + \langle  G_f(X_t, \dot X_t), \dot X_t\rangle \Delta t + o(\Delta t).
\end{align}
The second term, $D_h(x, X_t + \frac{t}{p} \dot X_t)$, is differentiable, with derivative $-\left\langle \frac{d}{dt}\nabla h(Z_t) , x - Z_t\right\rangle$. Hence,
\begin{align*}
&D_h\Big(x, X_{t+\Delta t} + \frac{t+\Delta t}{p} \dot X_{t+\Delta t}\Big) - D_h\Big(x, X_{t} + \frac{t}{p} \dot X_t\Big) \\
&=-\left\langle \frac{d}{dt} \nabla h(Z_t), x - Z_t\right\rangle \Delta t + o(\Delta t)\\
&= pt^{p-1}\langle G_\varphi(X_t, \dot X_t) + G_\psi(Z_t, \dot Z_t)), x - Z_t\rangle  \Delta t + o(\Delta t)\\
&=p t^{p-1} \langle G_\varphi(X_t, \dot X_t), x - X_t\rangle  \Delta t + t^p\langle G_\varphi(X_t, \dot X_t),\dot X_t\rangle \Delta t  + pt^{p-1}\langle  G_\psi(Z_t, \dot Z_t)), x - Z_t\rangle \Delta t + o(\Delta t)\\
&\leq -p t^{p-1}(\varphi(X_t) - \varphi(x))  \Delta t + t^p\langle G_\varphi(X_t, \dot X_t),\dot X_t\rangle \Delta t  -p t^{p-1}(\psi(Z_t) - \psi(x))\Delta t + o(\Delta t)\\
&\leq -p t^{p-1}(\varphi(X_t) - \varphi(x))  \Delta t + t^p\langle G_\varphi(X_t, \dot X_t),\dot X_t\rangle \Delta t  -p t^{p-1}(\psi(X_t) - \psi(x))\Delta t\\
& \ \ + t^p \langle G_\psi(X_t, \dot X_t), \dot X_t\rangle \Delta t + o(\Delta t)\\
&= -p t^{p-1}(f(X_t) - f(x))  \Delta t + t^p\langle G_f(X_t, \dot X_t),\dot X_t\rangle \Delta t.
\end{align*}
The last two inequalities follows from the convexity of $f = \varphi + \psi$. In the last inequality, we have used the identity $ Z_t - X_t= \frac{t}{p}\dot X_t$ in the term $pt^{p-1}\langle G_\psi(X_t, Z_t - X_t), Z_t - X_t\rangle$. Combining everything we have shown
\begin{align*}
\lim \sup_{\Delta t \rightarrow 0^+} \frac{\mathcal{E}_{t + \Delta t} - \mathcal{E}_{t }}{\Delta t} \leq 0,
\end{align*} 
which along with the continuity of $\mathcal{E}_{t }$, ensures $\mathcal{E}_{t}$ is a non-increasing of time. 
\end{proof}

\paragraph{Algorithm.}
Now we will discretize the dynamics~\eqref{eq:prox_dyn}. We assume the ideal scaling~\eqref{Eq:IdeScaBet} holds with equality. Using the same identifications $\dot X_t = \frac{x_{k+1} - x_k}{\delta}$, $\frac{d}{dt} \nabla h(Z_t) = \frac{\nabla h(z_{k+1}) - \nabla h(z_k)}{\delta}$ and $\frac{d}{dt}e^{\beta_t} = \frac{A_{k+1} - A_k}{\delta}$ , we apply the implicit-Euler scheme to~\eqref{eq:comp} and the explicit-Euler scheme to \eqref{eq:zup}. Doing so, we obtain a proximal mirror descent update, 
\begin{align*}
z_{k+1} = \arg \min_{z\in \X} \left\{  \psi(z) + \langle \nabla \varphi(x_{k+1}), z\rangle + \frac{1}{\alpha_k}D_h(z, z_k) \right\},
\end{align*}
and the sequence \eqref{eq:coup}, respectively. We write the algorithm as 
\begin{subequations}\label{eq:algprox1}
\begin{align}
x_{k+1} &= \tau_k z_k + (1- \tau_k) y_k \label{eq:coup8}\\
\nabla h(z_{k+1}) - \nabla h(z_k) &= -\alpha_k \nabla \varphi(x_{k+1}) - \alpha_k \nabla \psi(z_{k+1})\label{eq:proxstep}\\
y_{k+1} &= \mathcal{G}(x)\label{eq:grad8},
\end{align}
\end{subequations}
where we have similarly substituted the state $x_k$ with a sequence $y_k$, and added the update $y_{k+1} = \mathcal{G}(x)$. 
We summarize how the initial bound scales for algorithm~\eqref{eq:algprox1} in the following proposition. 
\begin{proposition}\label{prop:fista}
Assume $h$ is strongly convex, $\varphi$ is  $(1/\epsilon)$-smooth and $\psi$ is simple but not necessarily smooth. Using the Lyapunov function~\eqref{eq:lyap}, the following initial bound
\begin{align*}
E_{k+1} - E_k  \leq \varepsilon_{k+1},
\end{align*}
can be shown for algorithm~\eqref{eq:algprox1}, where the error scales as
\begin{align*}
\varepsilon_{k+1} = - \frac{\sigma}{2}\|z_{k+1} - z_{k}\|^2 &+ \ \frac{A_{k+1}}{2\epsilon}\|\tau_k z_k + (1- \tau_k) y_k - y_{k+1}\|^2  + A_{k+1} \psi(y_{k+1}) - A_k \psi(y_k) - \alpha_k\psi(z_{k+1}) .
\end{align*}
\end{proposition}

Tseng~\cite[Algorithm 1]{Tseng08}  showed that the map
\begin{align}\label{eq:weirdG}
\mathcal{G}(x) = \tau_{k} z_{k+1} + (1- \tau_{k}) y_{k}
\end{align}
can be used to simplify the error to the following, 
\begin{align}\label{eq:cond82}
\varepsilon_{k+1} = - \frac{\sigma}{2}\|z_{k+1} - z_k\|^2 + \frac{A_{k+1}\tau_k^2}{2\epsilon}\|z_{k+1} - z_k\|^2.
\end{align}
Notice that the condition necessary for the error to be non-positive
is the same as the condition for accelerated gradient descent~\eqref{eq:gradbound}. Using the same polynomial, we can conclude an $O(1/\epsilon\sigma k^2)$ convergence rate.
\begin{proof}
We begin with the observation that the update~\eqref{eq:weirdG} and the convexity of $\psi$ allow us to show the inequality
\begin{align}\label{eq:proxiden}
 A_{k+1}\psi((1- \tau_k) y_k + \tau_k z_{k+1})&\leq A_{k+1}(1- \tau_k)\psi( y_k) + A_{k+1}\tau_k \psi( z_{k+1})
\end{align}
Thus we can conclude $A_{k+1}\psi(y_{k+1}) -A_k \psi(y_k) \leq \alpha_k\psi( z_{k+1})$. With this, the standard Lyapunov analysis follows:
\begin{align*}
E_{k+1} - E_k & = D_h(x, z_{k+1})  -  D_h(x, z_{k}) + A_{k+1}(f(y_{k+1}) - f(x)) - A_{k}(f(y_{k}) - f(x))\\ 
&\leq  D_h(x, z_{k+1})  -  D_h(x, z_{k}) + A_{k+1}(\varphi(y_{k+1}) - \varphi(x)) - A_{k}(\varphi(y_{k}) - \varphi(x)) +  \alpha_k(\psi(z_{k+1}) - \psi(x))\\
&=   \alpha_k \langle \nabla \varphi(x_{k+1}) + \nabla \psi(z_{k+1}), x - z_{k+1}\rangle  -  D_h(z_{k+1}, z_{k})  \\
&\quad+ A_{k+1}(\varphi(y_{k+1}) - \varphi(x)) - A_{k}(\varphi(y_{k}) - \varphi(x))+  \alpha_k(\psi(z_{k+1}) - \psi(x))\\
&\leq \alpha_k \langle \nabla \varphi(x_{k+1}), x - z_{k}\rangle + \alpha_k \langle \nabla \varphi(x_{k+1}), z_k - z_{k+1}\rangle  -  D_h(z_{k+1}, z_{k}) \\
&\quad+ A_{k+1}(\varphi(y_{k+1}) - \varphi(x)) - A_{k}(\varphi(y_{k}) - \varphi(x))\\
&=\alpha_k \langle \nabla \varphi(x_{k+1}), x - z_{k}\rangle + A_{k+1} \langle \nabla \varphi(x_{k+1}), x_{k+1} - y_{k+1}\rangle  -  D_h(z_{k+1}, z_{k}) \\
&\quad+ A_{k+1}(\varphi(y_{k+1}) - \varphi(x)) - A_{k}(\varphi(y_{k}) - \varphi(x))\\
&\leq \alpha_k \langle \nabla \varphi(x_{k+1}), x - z_{k}\rangle + \frac{A_{k+1}}{2\epsilon} \|x_{k+1} - y_{k+1}\|^2  -  D_h(z_{k+1}, z_{k}) \\
&\quad+ \alpha_k(\varphi(x_{k+1}) - \varphi(x)) + A_{k}(\varphi(x_{k+1}) - \varphi(y_{k})).
\end{align*}
The first inequality uses the identity~\eqref{eq:proxiden}. The second inequality follows from the convexity of $\psi$. The last line uses the $\frac{1}{\epsilon}$-smoothness of $\varphi$.
It simply remains to use the $\sigma$-strong convexity of $h$ and the identities~\eqref{eq:coup} and $x_{k+1} - y_{k+1} = \tau_k(z_{k+1} - z_k)$. Continuing from the last line, and using these properties, we have 
\begin{align*}
E_{k+1} - E_k &\leq \alpha_k \langle \nabla \varphi(x_{k+1}), x - x_{k+1}\rangle + \frac{A_{k+1}\tau_k^2}{2\epsilon} \|z_{k+1} - z_{k}\|^2  -  \frac{\sigma}{2}\|z_{k+1}- z_{k}\|^2 \\
&\quad+ \alpha_k(\varphi(x_{k+1}) - \varphi(x)) + A_{k}(\varphi(x_{k+1}) - \varphi(y_{k})+ \langle \nabla \varphi(x_{k+1}), y_k - x_{k+1}\rangle)\\
&\leq \frac{A_{k+1}\tau_k^2}{2\epsilon} \|z_{k+1} - z_{k}\|^2  -  \frac{\sigma}{2}\|z_{k+1}- z_{k}\|^2. 
\end{align*}
The last line follows from the convexity of $\varphi$. 
\end{proof}


\subsubsection{Strongly convex functions}
We study the problem of minimizing the composite objective $f = \varphi + \psi$
in the setting where $ \varphi$ is $(1/\epsilon)$-smooth and $\mu$-strongly convex and $\psi$ is simple but not smooth. Like the setting where $f$ is weakly convex, we begin with the following proposition concerning dynamics that are relevant for this setting. 
\begin{proposition}\label{prop:prox_dyn2}
Define $f = \varphi + \psi$ and assume $\varphi$ is $\mu$-strongly convex with respect to $h$ and $\psi$ is convex. Under the ideal scaling condition~\eqref{Eq:IdeScaBet}, Lyapunov function~\eqref{eq:Lyap2} can be used to show that solutions to dynamics,
\begin{subequations}\label{eq:prox_dyn2}
\begin{align}
Z_t &= X_t + e^{-\alpha_t} \dot X_t \label{eq:zup2}\\
\frac{d}{dt} \nabla h(Z_t)  &= \dot \beta_t\nabla h(X_t) - \dot \beta_t \nabla h(Z_t) - \frac{e^{\alpha_t}}{\mu} (\nabla \varphi(X_t)+ \nabla \psi(Z_t)),\label{eq:comp2}
\end{align}
\end{subequations}
satisfy $f(X_t) - f(x) \leq O(e^{-\beta_t})$.
\end{proposition}
\begin{proof}
\begin{align*}
\frac{d}{dt}\left\{\mu e^{\beta_t} D_h(x, Z_t)\right\} &= \mu \dot \beta_t e^{\beta_t} D_h(x, Z_t) - \mu e^{\beta_t}\left\langle \frac{d}{dt}\nabla  h(Z_t), x - Z_t\right\rangle \\
&= \mu \dot \beta_t e^{\beta_t}\Big(\left\langle\nabla h(Z_t) - \nabla h(X_t), x - Z_t\right\rangle + D_h(x, Z_t)\Big)\\
&\quad  + e^{\alpha_t+ \beta_t} \langle \nabla \varphi(X_t) + \nabla \psi(Z_t), x- Z_t\rangle\\
&= \mu \dot \beta_t e^{\beta_t}\Big( D_h(x, X_t) - D_h(Z_t, X_t)\Big)  + \dot \beta_te^{ \beta_t} \langle \nabla \varphi(X_t),  x- X_t\rangle + e^{\beta_t} \langle \nabla \varphi(X_t), \dot X_t\rangle\\ 
& \quad + e^{\beta_t}\left(e^{\alpha_t} - \dot \beta_t\right)\langle \nabla \varphi(X_t),  x- X_t\rangle + e^{\alpha_t+ \beta_t} \langle\nabla \psi(Z_t), x- Z_t\rangle.
\end{align*}
The second line comes from plugging in dynamics \eqref{eq:comp2}. The third line uses the Bregman three-point identity~\eqref{eq:bregthree}. We continue by using the strong convexity assumption:
\begin{align*}
\frac{d}{dt}\left\{\mu e^{\beta_t} D_h(x, Z_t)\right\} &= \leq -\dot \beta_t e^{\beta_t} (\varphi(X_t)  -\varphi(x))- e^{\beta_t} \langle \nabla \varphi(X_t), \dot X_t\rangle 
+ e^{\beta_t}\left(e^{\alpha_t} - \dot \beta_t\right) \langle \nabla \varphi(X_t),  x- X_t\rangle \\
&\quad- e^{\alpha_t+ \beta_t} (\psi(Z_t)-  \psi(x))\\
&\leq -\dot \beta_t e^{\beta_t} (f(X_t)  -f(x))- e^{\beta_t} \langle \nabla \varphi(X_t), \dot X_t\rangle  - e^{\beta_t + \alpha_t} \langle \nabla \psi(X_t),Z_t - X_t\rangle \\
&\quad+ e^{\beta_t}\left(e^{\alpha_t} - \dot \beta_t\right)\left( \langle \nabla \varphi(X_t),  x- X_t\rangle - (\psi(X_t)-  \psi(x))\right) \\
&\leq -\dot \beta_t e^{\beta_t} (f(X_t)  -f(x))- e^{\beta_t} \langle \nabla f(X_t), \dot X_t\rangle\\
& + e^{\beta_t}\left(e^{\alpha_t} - \dot \beta_t\right)\left( \langle \nabla \varphi(X_t),  x- X_t\rangle - (\psi(X_t)-  \psi(x))\right) \\
&\leq -\frac{d}{dt}\left\{e^{\beta_t} (f(X_t)  -f(x))\right\}.
\end{align*}
 The fourth line follows the strong convexity of $\varphi$ and convexity of $\psi$. The fifth line (second inequality) uses the convexity of $\psi$ once again. The third inequality plugs in the definition of $Z_t - X_t$ and the second-last inequality follows from the chain rule and the ideal scaling condition~\eqref{Eq:IdeScaBet}.
\end{proof}

 Assume $h$ is Euclidean and the ideal scaling~\eqref{Eq:IdeScaBet} holds with equality $\dot \beta_t = e^{\alpha_t}$. To discretize the dynamics~\eqref{eq:comp2}, we split the vector field~\eqref{eq:comp2} into two components, $v_1(x,z,t) =  \dot \beta_t (X_t - Z_t - (1/\mu) \nabla \varphi(X_t))$, and $v_2(x,z,t) = - \dot \beta_t/ \mu \nabla \psi(Z_t)$ and apply the explicit Euler scheme to $v_2(x,z,t)$ and the implicit Euler scheme to $v_1(x,z,t)$, with the same identification $\dot \beta_t = \tau_k/\delta$ for both vector fields.\footnote{While using the same identification of $\dot \beta_t$ for both vector fields is problematic---since one is being evaluated forward in time and the other backward in time---the error bounds only scale sensibly in the setting where $\dot \beta_t = \gamma \leq \sqrt{\mu}$ is a constant. }
This results in the proximal update
\begin{align}\label{eq:proxstrong}
z_{k+1} = \arg \min_z \left\{  \psi(z) + \langle \nabla \varphi(x_{k}), z\rangle + \frac{\mu}{2\tau_k}\|z - (1- \tau_k)z_k - \tau_kx_k\|^2 \right\}.
\end{align}
In full, we can write the algorithm as 
\begin{subequations}\label{eq:algprox2}
\begin{align} 
x_k &= \frac{\tau_k}{1+\tau_k} z_k + \frac{1}{1+\tau_k} y_k \label{Eq:Coupling122}\\
z_{k+1} - z_k &= \tau_k \left(x_k - z_k - \frac{1}{\mu}\nabla \varphi(x_k) - \frac{1}{\mu}\nabla \psi(z_{k+1})\right) \label{eq:prox_zstep}\\
y_{k+1} &= \mathcal{G}(x). \label{Eq:Grad122}
\end{align}
\end{subequations}
We summarize how the initial bound changes with this modified update in the following proposition.
\begin{proposition}\label{prop:fista3}
Assume $h$ is Euclidean, $\varphi$ is strongly convex, $\varphi$ is  $(1/\epsilon)$-smooth, and $\psi$ is convex and simple.  
Using the Lyapunov function~\eqref{eq:lyapdd}, we have
\begin{align*}
E_{k+1} - E_k  \leq \varepsilon_{k+1},
\end{align*}
for algorithm~\eqref{eq:algprox2}, where
\begin{align*}
\varepsilon_{k+1} &= - A_{k+1}\frac{\mu}{2}\| (z_k - z_{k+1}) - \tau_k(z_k - x_k)\|^2 + \frac{A_{k+1}}{2\epsilon}\|x_{k} - y_{k+1}\|^2\\
&\qquad+ A_{k+1}\left(\frac{\tau_k}{2\epsilon} - \frac{\mu}{2\tau_k}\right)\|x_k - y_k\|^2 + A_{k+1}\psi(y_{k+1}) - A_k \psi(y_k) - \alpha_k \psi(z_{k+1}).\notag 
\end{align*}
\end{proposition}
Using the same update~\eqref{eq:weirdG},
\begin{align*}
\mathcal{G}(x) = \tau_{k} z_{k+1} + (1- \tau_{k}) y_{k},
\end{align*}
the bound simplifies nicely,
\begin{align}\label{eq:errprox}
\varepsilon_{k+1} /A_{k+1}&= \left(\frac{\tau_k^2}{2\epsilon}- \frac{\mu}{2}\right)\frac{1}{\mu}\| \nabla \varphi(x_k) + \nabla \psi(z_{k+1})  \|^2  + \left(\frac{\tau_k}{2\epsilon} - \frac{\mu}{2\tau_k}\right)\|x_k - y_k\|^2\notag.
\end{align}
The condition necessary for the error to be non-positive,
$\tau_k \leq \sqrt{\epsilon \mu} = 1/\sqrt{\kappa},
$ 
results in a $O(e^{-k/\sqrt{\kappa}})$ convergence rate. This matches the lower bound for the class of $(1/\epsilon)$-smooth and $\mu$-strongly convex functions.  As in continuous time, this analysis also allows for the use of subgradients of $\psi$. 
\begin{proof}

\begin{align*}
\tilde E_{k+1} - \tilde E_k &= \frac{\mu}{2}\|x^\ast- z_{k+1}\|^2 -  \frac{\mu}{2}\|x^\ast -  z_{k}\|^2 + f(y_{k+1}) - f(y_k)\\
&= -\mu \langle z_{k+1} - z_k, x^\ast -  z_{k+1}\rangle - \frac{\mu}{2}\|z_{k+1} - z_k\|^2  + f(y_{k+1}) - f(y_k)\\
& \leq - \mu \langle z_{k+1} - z_k, x^\ast -  z_{k+1}\rangle - \frac{\mu}{2}\|z_{k+1} - z_k\|^2 +  \langle \nabla \varphi (x_k), y_{k+1} - y_k\rangle  + \frac{1}{2\epsilon}\|x_{k} - y_{k+1}\|^2\\
&\quad - \frac{\mu}{2}\|x_k -y_k\|^2  -\tau_k \psi(y_k) - \tau_k \psi(z_{k+1})\\
 &\overset{\eqref{eq:weirdG}}{=} - \mu \langle z_{k+1} - z_k, x^\ast -  z_{k+1}\rangle - \frac{\mu}{2}\|z_{k+1} - z_k\|^2 +  \tau_k \langle \nabla \varphi(x_k), z_{k+1} - y_k\rangle  - \frac{\mu}{2}\|x_k -y_k\|^2\\
&\quad + \frac{1}{2\epsilon}\|x_{k} - y_{k+1}\|^2  -\tau_k \psi(y_k) - \tau_k \psi(z_{k+1})\\
&\overset{\eqref{eq:proxstep}}{=} \tau_k \langle \nabla \varphi(x_k), x^\ast -  z_{k+1}\rangle - \frac{\mu}{2}\|z_{k+1} - z_k\|^2 +  \tau_k \langle \nabla \varphi (x_k), z_{k+1} - y_k\rangle + \frac{1}{2\epsilon}\|x_{k} - y_{k+1}\|^2 \\
&\quad  - \frac{\mu}{2}\|x_k -y_k\|^2 + \mu \tau_k\langle x_k - z_k, x^\ast- z_{k+1}\rangle + \tau_k \langle \nabla \psi(z_{k+1}), x^\ast -  z_{k+1}\rangle -\tau_k \psi(y_k) - \tau_k \psi(z_{k+1}) \\
&\leq \tau_k \langle \nabla \varphi(x_k), x^\ast -  x_k\rangle - \frac{\mu}{2}\|z_{k+1} - z_k\|^2 +  \tau_k \langle \nabla \varphi(x_k), x_k - y_k\rangle  + \frac{1}{2\epsilon}\|x_{k} - y_{k+1}\|^2\\
&\quad - \frac{\mu}{2}\|x_k -y_k\|^2 + \mu \tau_k\langle x_k - z_k, x^\ast- z_{k+1}\rangle  -\tau_k (\psi(y_k) - \psi (x^\ast))\\
%
&\leq - \tau_k\left(\varphi(x_k) - \varphi(x^\ast) + \frac{\mu}{2}\|x^\ast - x_k\|^2\right) + \mu \tau_k\langle x_k - z_k, x^\ast- z_{k}\rangle + \tau_k \langle \nabla \varphi (x_k), x_k -  y_k\rangle  \\
&\quad + \frac{1}{2\epsilon}\|x_{k} - y_{k+1}\|^2 - \frac{\mu}{2}\|x_k -y_k\|^2 + \mu \tau_k\langle x_k - z_k, z_{k} - z_{k+1}\rangle  -\frac{\mu}{2}\|z_{k+1} - z_k\|^2\\
&\quad -\tau_k (\psi(y_k) - \psi(x^\ast)).
\end{align*}
The first inequality follows from the strong convexity and $(1/\epsilon)$-smoothness of $\varphi$ and~\eqref{eq:proxiden}, from which we can conclude $\varphi(y_{k+1}) - \varphi(y_k) \leq -\tau_k \varphi(y_k) - \tau_k \varphi(z_{k+1})$. The second inequality  follows from the convexity of $\psi$. The third inequality uses the strong convexity of $f$. Next, we use identity~\eqref{eq:bregthree} and the smoothness of $\varphi$ to simplify the bound as follows:
\begin{align*}
\tilde E_{k+1} - \tilde E_k
%
%
& \overset{\eqref{Eq:Coupling1}}{\leq} - \tau_k\left(\varphi(x_k) - f(x^\ast) + \frac{\mu}{2}\|x^\ast - z_k\|^2\right) - \frac{ \mu }{2\tau_k}\|x_k - y_k\|^2 + \tau_k \langle \nabla \varphi(x_k), x_k -  y_k\rangle \\
&\quad + \frac{1}{2\epsilon}\|x_{k} - y_{k+1}\|^2 - \frac{\mu}{2}\|x_k -y_k\|^2 + \mu \tau_k\langle x_k - z_k, z_{k} - z_{k+1}\rangle  -\frac{\mu}{2}\|z_{k+1} - z_k\|^2  -\tau_k \psi(y_k)\\
&\leq - \tau_k\left(f(y_k) - f(x^\ast) + \frac{\mu}{2}\|x^\ast - z_k\|^2\right) - \frac{\mu}{2\tau_k}\|x_k - y_k\|^2+ \frac{\tau_k}{2\epsilon} \| x_k -  y_k\|^2 \\
&\quad + \frac{1}{2\epsilon}\|x_{k} - y_{k+1}\|^2 - \frac{\mu}{2}\|x_k -y_k\|^2 + \mu \tau_k\langle x_k - z_k, z_{k} - z_{k+1}\rangle  -\frac{\mu}{2}\|z_{k+1} - z_k\|^2 \\
&\overset{\eqref{Eq:Coupling1}}{=}- \tau_k E_k - \frac{\tau_k^2 \mu}{2}\|x_k - z_k\|^2  + \frac{1}{2\epsilon}\|x_{k} - y_{k+1}\|^2 + \mu \tau_k\langle x_k - z_k, z_{k} - z_{k+1}\rangle \\
&\quad -\frac{\mu}{2}\|z_{k+1} - z_k\|^2 + \left(\frac{\tau_k}{2\epsilon} - \frac{\mu}{2\tau_k}\right)\|x_k - y_k\|^2\\
&=- \tau_k E_k - \frac{\mu}{2}\|\tau_k(x_k - z_k) - (z_k - z_{k+1})\|^2+ \frac{1}{2\epsilon}\|x_{k} - y_{k+1}\|^2  + \left(\frac{\tau_k}{2\epsilon} - \frac{\mu}{2\tau_k}\right)\|x_k - y_k\|^2.
\end{align*}
It remains to check \begin{align*}x_k - y_{k+1} \overset{\eqref{eq:weirdG}}{=} x_k - y_k - \tau_k(z_{k+1} - y_k)&\overset{\eqref{Eq:Coupling1}}{=} \tau_k(z_k - x_k - z_{k+1} + y_k) \overset{\eqref{Eq:Coupling1}}{=}  \tau_k(\tau_k(x_k - z_k) - (z_k - z_{k+1})).\end{align*} \end{proof}

   \subsection{Stochastic methods}
   \label{App:Stoch}
    We begin with the following proposition.
 \begin{claim}
  Assume $h$ is $\sigma$-strongly convex and $f$ is convex. For algorithm~\eqref{eq:agd}, where stochastic gradients are used instead of full gradients and $\mathcal{G}(x) = x_{k+1}$ , we can show the following error bound:
   \begin{align}
\frac{\mathbb{E}[E_{k+1}]- E_{k}}{\delta} \leq \mathbb{E}[\varepsilon_{k+1}],
\end{align}
for Lyapunov function~\eqref{eq:nlyap}, where the error scales as
 \begin{align}\label{eq:erboud1}
\mathbb{E}[\varepsilon_{k+1}] = \frac{(A_{k+1}- A_k)^2}{2\sigma\delta}\mathbb{E}[\|G(x_{k+1})\|^2].
\end{align}
For algorithm~\eqref{eq:str1}, where stochastic gradients are used instead of full gradients, we can show the following error bound:
\begin{align*}
\frac{\mathbb{E}[E_{k+1}]- E_{k}}{\delta} \leq -\frac{\tau_k}{\delta} E_k + \mathbb{E}[\varepsilon_{k+1}]
\end{align*}
for Lyapunov function~\eqref{eq:lyapdd}, where the error scales as
 \begin{align}\label{eq:secc}
\mathbb{E}[\varepsilon_{k+1}] =  \frac{ A_k\tau_{k}^2}{2\mu\sigma\delta}\mathbb{E}[\|G(x_{k+1})\|^2].
\end{align}
   \end{claim}
 The proof of this claim follows from the proof of Proposition~\ref{Prop:WeakBound} and~\ref{prop:quasi-strong}, where we simply take $\nabla f$ to be stochastic.
Maximizing over this sequence gives a $O(1/\sqrt{k})$ for the first algorithm and $O(1/k)$ for the second. This convergence rate is optimal and matches the rate of SGD. Notice, however, that the convergence rate is for the entire sequence of iterates, unlike SGD. 

\paragraph{Stochastic dynamics.}
Having introduced the dynamics~\eqref{eq:el2}, it is clear that the following stochastic dynamics
\begin{align*}
\mathrm{d}Z_t &= \dot \beta_t(X_t\mathrm{d}t - Z_t \mathrm{d}t- (1/\mu)(\nabla f(X_t) \mathrm{d}t+ \sigma(X_t, t) \mathrm{d}B_t))\\
\mathrm{d}X_t &= \dot \beta_t (Z_t - X_t)\mathrm{d}t
\end{align*}
is a natural candidate for {\em approximating} the stochastic variants of algorithms~\eqref{eq:str1} and~\eqref{Eq:StrongConv21} in the setting where $h(x) = \frac{1}{2}\|x\|^2$ is Euclidean and $f$ is $\mu$-strongly convex.\footnote{Some of the following statements can be made more rigorous and motivates further study. The dynamics can also be generalized to the more general setting in the natural way, but for simplicity we take $h$ to be Euclidean.} Here $B_t \in \mathbb{R}^d$ is a standard Brownian motion, and $\sigma(X_t,t)\in \mathbb{R}^{d \times d}$ is the diffusion coefficient. We assume $\mathbb{E}\|\sigma(X_t, t)^\top \sigma(X_t, t)\|^2 \leq M$ for some positive constant $M \in\mathbb{R}^{+}$.
We can use It\^o's formula to calculate $\mathrm{d}\mathcal{E}_t$ where $\mathcal{E}_t$ is given by \eqref{eq:breg2} as follows, 
\begin{align}\label{eq:ito}
\mathrm{d}\mathcal{E}_t = \frac{\partial \mathcal{E}_t}{\partial t} \mathrm{d}t + \left\langle\frac{ \partial \mathcal{E}_t}{\partial X_t}, \mathrm{d}X_t \right\rangle + \left\langle \frac{\partial \mathcal{E}_t}{\partial Z_t}, \mathrm{d}Z_t \right\rangle  + \frac{\dot \beta_t^2 e^{\beta_t}}{2 \mu} \text{tr} \left( \sigma(X_t, t)^\top \sigma(X_t, t) \right) \mathrm{d}t.
\end{align}
 We compute:
\begin{align*}
\frac{\partial \mathcal{E}_t}{\partial t} &= \dot \beta_t e^{\beta_t}\left( f(X_t) - f(x^\ast) + \frac{\mu}{2}\|x^\ast - Z_t\|^2\right),\\
\frac{\partial \mathcal{E}_t}{\partial X_t} &=e^{\beta_t}\nabla f(X_t),\\
\frac{\partial \mathcal{E}_t}{\partial Z_t} & = \mu(x^\ast - Z_t).
\end{align*}
Plugging this into~\eqref{eq:ito}, we have 
\begin{align*}
\mathrm{d}\mathcal{E}_t &= \dot \beta_t e^{\beta_t}\left( f(X_t) - f(x^\ast) + \frac{\mu}{2}\|x^\ast - Z_t\|^2\right) \mathrm{d}t  + \dot \beta_t e^{\beta_t} \left\langle\nabla f(X_t), Z_t - X_t \right\rangle \mathrm{d}t + \dot \beta_te^{\beta_t} \mu \langle x^\ast - Z_t,X_t  - Z_t\rangle \mathrm{d}t\\
& +  \dot \beta_te^{\beta_t}\langle x^\ast - Z_t,\nabla f(X_t)\rangle  \mathrm{d}t+ \dot \beta_te^{\beta_t}\langle x^\ast - Z_t,\sigma(X_t,t)\rangle\mathrm{d}t + \frac{\dot \beta_t^2 e^{\beta_t}}{2 \mu} \text{tr} \left( \sigma(X_t, t)^\top \sigma(X_t, t) \right) \mathrm{d}t\\
&\leq  \dot \beta_t\langle x^\ast - Z_t,\sigma(X_t,t)\rangle \mathrm{d}t+   \frac{\dot \beta_t^2 e^{\beta_t}}{2 \mu} \text{tr} \left( \sigma(X_t, t)^\top \sigma(X_t, t) \right) \mathrm{d}t,
\end{align*}
where the inequality follows from the proof of proposition~\ref{prop:strLyap} which can be found in Appendix~\ref{app:strLyapproof}. That is, we can conclude
\begin{align*}
\mathcal{E}_t \leq \mathcal{E}_0 -\int_0^t  \dot \beta_se^{\beta_s}\langle x^\ast - Z_s,\sigma(X_s,s)\rangle \mathrm{d}s+  \int_0^t  \frac{\dot \beta_s^2 e^{\beta_s}}{2 \mu} \text{tr} \left( \sigma(X_s, s)^\top \sigma(X_s, s) \right) \mathrm{d}s.
\end{align*}
If we take the expectation of both sides, the middle term, $\int_0^t  \dot \beta_se^{\beta_s}\langle x^\ast - Z_s,\sigma(X_s,s)\rangle \mathrm{d}s $, will vanish by the martingale property of the It\^o integral. This allows us to conclude that
\begin{align*}
\mathbb{E}[f(X_t) - f(x^\ast)] \leq \frac{\mathcal{E}_0 + \mathbb{E}\left[\int_0^t  \frac{\dot \beta_s^2 e^{\beta_s}}{2 \mu} \text{tr} \left( \sigma(X_s, s)^\top \sigma(X_s, s) \right) \mathrm{d}s\right]}{e^{\beta_t}}.
\end{align*}
In particular, choosing $\beta_t = 2\log t + \log(1/2)$, we obtain a $O(1/t^2)+ O(1/t)$ convergence rate. We can compare this upper bound to the bound~\eqref{eq:secc} with the identifications $\dot \beta_t = \tau_k$ and $e^{\beta_t} = A_k$.  

\end{document}